\documentclass{article}
\usepackage{amssymb,amsmath,amsthm,graphicx,cite}

\def\qed{\hfill {\hbox{${\vcenter{\vbox{               
   \hrule height 0.4pt\hbox{\vrule width 0.4pt height 6pt
   \kern5pt\vrule width 0.4pt}\hrule height 0.4pt}}}$}}}

\def\utr{\, \underline{\triangleright}\, }
\def\otr{\, \overline{\triangleright}\, }

\newtheorem{theorem}{Theorem}
\newtheorem{lemma}[theorem]{Lemma}
\newtheorem{proposition}[theorem]{Proposition}

\newtheorem{observation}{Observation}
\theoremstyle{definition}
\newtheorem{example}{Example}
\newtheorem{definition}{Definition}

\date{}

\title{\Large \textbf{Trace Diagrams and Biquandle Brackets}}

\author{Sam Nelson\footnote{Email: Sam.Nelson@cmc.edu. Partially supported by Simons Foundation collaboration grant 316709.}\and
Natsumi Oyamaguchi\footnote{Email: natsumi.3-29.math@diary.ocn.ne.jp.}}

\begin{document}
\maketitle

\begin{abstract}
We introduce a method of computing biquandle brackets of oriented knots
and links using a type of decorated trivalent spatial graphs we call
\textit{trace diagrams}. We identify algebraic conditions on the biquandle
bracket coefficients for moving strands over and under traces and identify
a new stop condition for the recursive expansion. In the case of monochromatic 
crossings we show that biquandle brackets satisfy a Homflypt-style skein 
relation and we identify algebraic conditions on the biquandle bracket 
coefficients to allow pass-through trace moves.
\end{abstract}

\parbox{6in} {\textsc{Keywords:} Quantum enhancements, biquandles, biquandle
counting invariants, biquandle brackets, trace diagrams

\smallskip

}

\section{\large\textbf{Introduction}}\label{I}

\textit{Biquandles} are a type of algebraic structure with axioms motivated
by the Reidemeister moves in knot theory. More precisely, a biquandle is a set
with two operations with axioms chosen so that the number of biquandle 
\textit{colorings} or assignments of biquandle elements to the semiarcs in
an oriented knot or link diagram satisfying certain conditions at crossings
is the same for all diagrams of the given knot or link, thus defining a 
nonnegative integer-valued invariant of knots and links known as  
the \textit{biquandle counting invariant} $\Phi_X^{\mathbb{Z}}$. An 
\textit{enhancement} of the biquandle counting invariant is a stronger 
invariant $\Phi_X^{\phi}$ which specializes to $\Phi_X^{\mathbb{Z}}$ in some way, 
e.g. by taking a cardinality or evaluating at $u=1$, etc. 
See \cite{EN,FJK,KR} etc. for more about biquandles.

In \cite{NOR} a type of enhancement of $\Phi_X^{\mathbb{Z}}$ was defined
using skein relations with coefficients depending on the biquandle colors at
a crossing. This setup poses an obvious objection: smoothings break the 
biquandle coloring. In \cite{NOR} this problem is resolved by thinking
of the invariant in state-sum form, doing all smoothings simultaneously
to obtain states without biquandle colors. Generalizations of biquandle 
brackets have recently been described in papers such as
\cite{NOSY} and \cite{IM}.

In this paper we describe a method for computing biquandle bracket invariants 
recursively using a type of decorated trivalent spatial graph diagram
called a \textit{trace diagram} similar to diagrams used in e.g. \cite{DK}
by defining biquandle colorings
for these diagrams. We identify algebraic conditions on the biquandle bracket
coefficients which are necessary and sufficient to allow moving strands over,
under or through traces. As an application we show that biquandle brackets 
satisfy a skein relation similar to that of the Homflypt polynomial \cite{H}
at monochromatic crossings, and we give an example to illustrate how this
can be helpful for faster hand computations of the invariant.

The paper is organized as follows. In Section \ref{B} we review biquandles, 
biquandle brackets and the biquandle bracket invariant. In Section \ref{T}
we introduce trace diagrams, their biquandle colorings and our method for
computing biquandle bracket values recursively in terms of trace diagrams.
We identify algebraic conditions for a biquandle bracket to admit 
overcrossing trace moves, undercrossing trace moves or both. In Section 
\ref{M} we look at the special case of monochromatic crossings, identifying
a Homflypt-style skein relation satisfied by all biquandle brackets as well
as the algebraic conditions required for a biquandle bracket to admit 
pass-through trace moves at monochromatic crossings.
We conclude in Section \ref{Q} with some questions for future work.

\section{\large\textbf{Biquandles and Biquandle Brackets}}\label{B}

In this section we briefly review biquandles and biquandle brackets; see
\cite{EN,NOR} for more.

\begin{definition}
A \textit{biquandle} is a set $X$ with operations $\utr,\otr:X\to X$
satisfying for all $x,y,z\in X$
\begin{itemize}
\item[(i)] $x\utr x=x\otr x$,
\item[(ii)] The maps $\alpha_x,\beta_x:X\to X$ and $S:X\times X\to X\times X$
given by 
\[\alpha_x(y)=y\otr x,\quad \beta_x(y)=y\utr x\quad \mathrm{and}\quad
S(x,y)=(y\otr x, x\utr y)\]
are invertible, and
\item[(iii)] We have the \textit{exchange laws}
\[\begin{array}{rcl}
(x\utr y)\utr(z\utr y) & = & (x\utr z)\utr(y\otr z) \\
(x\utr y)\otr(z\utr y) & = & (x\otr z)\utr(y\otr z) \\
(x\otr y)\otr(z\otr y) & = & (x\otr z)\otr(y\utr z).
\end{array}\]
\end{itemize}
It is sometimes convenient for the sake of space to
write $x\utr y$ as $x^y$ and $x\otr y$ as $x_y$.
\end{definition}

\begin{example}
A useful class of biquandles is \textit{Alexander biquandles}: let $X$
be a module over the two-variable Laurent polynomial ring 
$\mathbb{Z}[t^{\pm 1},s^{\pm 1}]$ and define
\[x^y=tx+(s-t)y\quad\mathrm{and}\quad x_y=sx.\]
One easily verifies that the biquandle axioms are satisfied. In particular,
if $X=\mathbb{Z}_n$ and $s,t\in X$ are coprime to $n$, then $X$ is a 
\textit{linear} Alexander biquandle.
\end{example}

The biquandle axioms are chosen so that given a \textit{coloring} of an 
oriented link diagram $L$ by a biquandle $X$, i.e., an assignment of elements
of $X$ to the semiarcs of $L$ satisfying at every crossing the conditions
\[\includegraphics{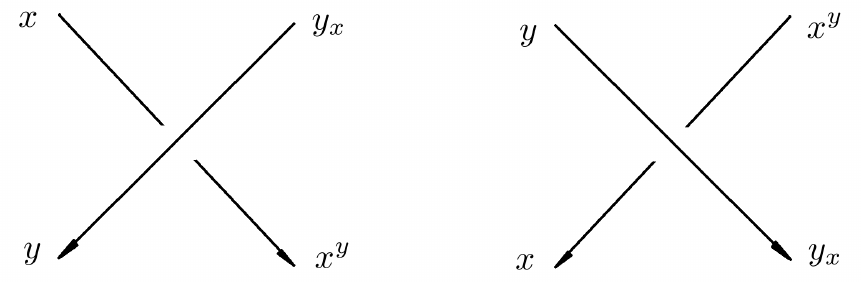}\]
before a Reidemeister move, there is a unique coloring of the diagram after 
the move which agrees with the pre-move coloring outside the neighborhood 
of the move. It follows that for a finite biquandle $X$, the number of 
$X$-colorings of an oriented link diagram is a link invariant. Indeed, it is
not just the \textit{number} of such colorings, but the \textit{set} of such 
colorings for any fixed choice of diagram or, equivalently, the set of 
equivalence classes of colorings of diagrams of $L$, which is an invariant of
links. In particular, if $X$ is a finite biquandle, then the set of colorings
of a (tame) oriented knot or link is finite and can be computed from a diagram,
either by brute force counting or by using the structure of the biquandle
where possible.

\begin{example}
Let us compute the set of colorings of the trefoil knot $3_1$
\[\includegraphics{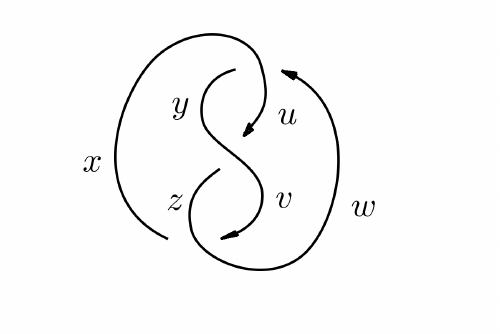}\]
by the linear Alexander biquandle $X=\mathbb{Z}_3$ with $t=1$ and $s=2$. 
We have biquandle operations
\[x^y=x+y\quad\mathrm{and}\quad x_y=2x\]
yielding system of coloring equations over $\mathbb{Z}_3$
\[
\begin{array}{rcl}
2x & = & u \\
x+y & = & w \\
2y & = & v \\
y+z & = & u \\
2z & = & w \\
x+z & = & v
\end{array}\]
so the space of colorings is the kernel of the matrix
\[
\left[\begin{array}{rrrrrr}
2 & 0 & 0 & 2 & 0 & 0 \\
1 & 1 & 0 & 0 & 0 & 2 \\
0 & 2 & 0 & 0 & 2 & 0 \\
0 & 1 & 1 & 2 & 0 & 0 \\
0 & 0 & 2 & 0 & 0 & 2 \\
1 & 0 & 1 & 0 & 2 & 0
\end{array}\right]
\stackrel{\mathrm{row-equiv.\ over\ }\mathbb{Z}_3}{\longleftrightarrow}
\left[\begin{array}{rrrrrr}
1 & 0 & 0 & 0 & 2 & 2 \\
0 & 1 & 0 & 0 & 1 & 0 \\
0 & 0 & 1 & 0 & 0 & 1 \\
0 & 0 & 0 & 1 & 1 & 1 \\
0 & 0 & 0 & 0 & 0 & 0 \\
0 & 0 & 0 & 0 & 0 & 0 \\
\end{array}\right]
\]
given by  $\mathbb{Z}[(1,2,0,2,1,0), (1,0,2,2,0,1)]$. Then 
$\Phi_X^{\mathbb{Z}}(3_1)=3^2=9$.

Alternatively, we can compute the set of colorings diagrammatically by
checking which assignments of elements of $X$ satisfy the coloring conditions
at every crossing.
\[\includegraphics{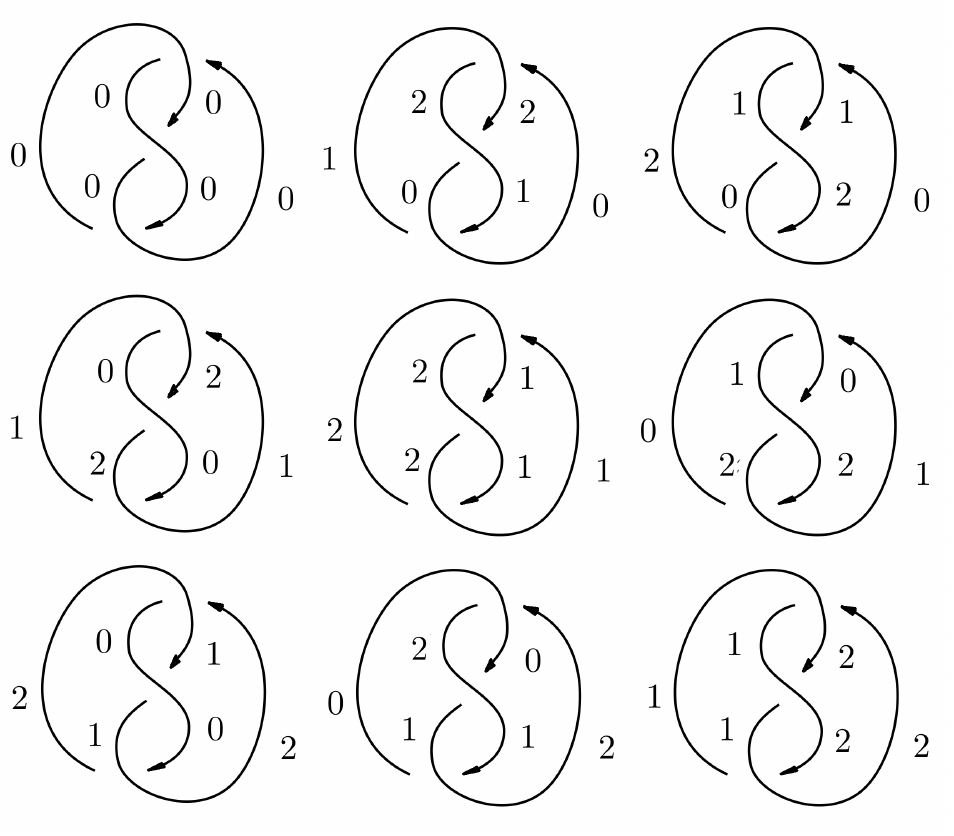}\]
\end{example}

\begin{definition}
A map $f:X\to Y$ between biquandles is a \textit{homomorphism} if we have
\[f(x\utr x')=f(x)\utr f(x') \quad \mathrm{and}\quad
f(x\utr x')=f(x)\otr f(x')\]
for all $x,x'\in X$.
\end{definition}

The set of $X$-colorings of a knot or link diagram $L$ can be identified
with the set of \textit{biquandle homomorphisms} $f:\mathcal{B}(L)\to X$
from the the \textit{fundamental biquandle} $\mathcal{B}(L)$ of $L$, defined as the biquandle generated by the semiarcs of $L$ modulo the crossing relations
of $L$, to the biquandle $X$. In particular, a coloring provides an
image $f(x_j)\in X$ for each generator $x_j$ of $\mathcal{B}(L)$, defining
a unique homomorphism $f:\mathcal{B}(L)\to X$ if and only if the crossing 
relations of $L$ are satisfied. Thus, we have 
$\Phi_X^{\mathbb{Z}}(L)=|\mathrm{Hom}(\mathcal{B}(L),X)|$. For more see \cite{EN}.

\begin{definition}
Now, let $X$ be a biquandle and $R$ a commutative ring with identity. A 
\textit{biquandle bracket} over $R$ is a pair of maps 
$A,B:X\times X\to R^{\times}$ assigning units $A_{x,y},B_{x,y} \in R$ to pairs 
of elements $(x,y)\in X\times X$ such that the following conditions are
satisfied:
\begin{itemize}
\item[(i)] For all $x\in X$, the elements $-A_{x,x}^2B_{x,x}^{-1}$ are equal, 
with their common value denoted  as $w\in R$;
\item[(ii)] For all $x,y\in X$, the elements $-A_{x,y}^{-1}B_{x,y}-A_{x,y}B_{x,y}^{-1}$
are equal, with their common value denoted as $\delta$, and
\item[(iii)] For all $x,y,z\in X$ we have
\[\begin{array}{rcl}
A_{x,y}A_{y,z}A_{x^y,z_y} & = & A_{x,z}A_{y_x,z_x}A_{x^z,y^z} \\
A_{x,y}B_{y,z}B_{x^y,z_y} & = & B_{x,z}B_{y_x,z_x}A_{x^z,y^z} \\
B_{x,y}A_{y,z}B_{x^y,z_y} & = & B_{x,z}A_{y_x,z_x}B_{x^z,y^z} \\
A_{x,y}A_{y,z}B_{x^y,z_y} & = & 
A_{x,z}B_{y_x,z_x}A_{x^z,y^z} 
+A_{x,z}A_{y_x,z_x}B_{x^z,y^z} \\ 
& & +\delta A_{x,z}B_{y_x,z_x}B_{x^z,y^z} 
+B_{x,z}B_{y_x,z_x}B_{x^z,y^z} \\
B_{x,y}A_{y,z}A_{x^y,z_y} 
+A_{x,y}B_{y,z}A_{x^y,z_y} & & \\
+\delta B_{x,y}B_{y,z}A_{x^y,z_y} 
+B_{x,y}B_{y,z}B_{x^y,z_y}  
& = & B_{x,z}A_{y_x,z_x}A_{x^z,y^z}. \\
\end{array}\]
\end{itemize} 
We can specify a biquandle bracket with an $n\times 2n$ block matrix $[A|B]$
where the $(i,j)$-entries of $A$ and $B$ respectively are $A_{x_i,x_j}$ and
$B_{x_i,x_j}$.
See \cite{NOR} for more.
\end{definition}

The biquandle bracket axioms are chosen so that the state-sum expansion of
an $X$-colored oriented knot or link diagram $L_f$ using the skein relations
\[\includegraphics{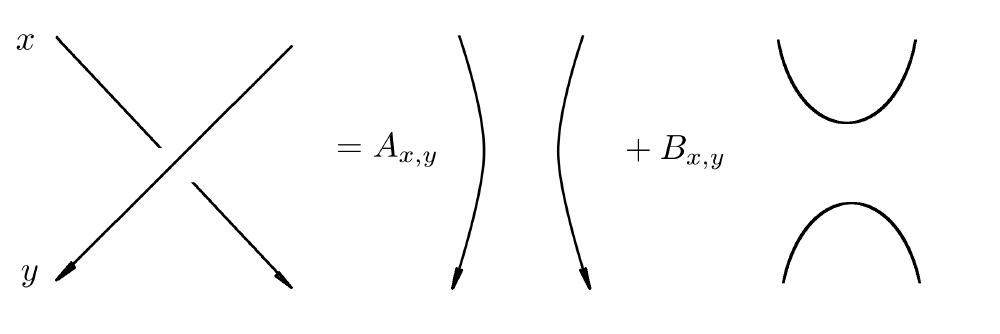}\]
\[\includegraphics{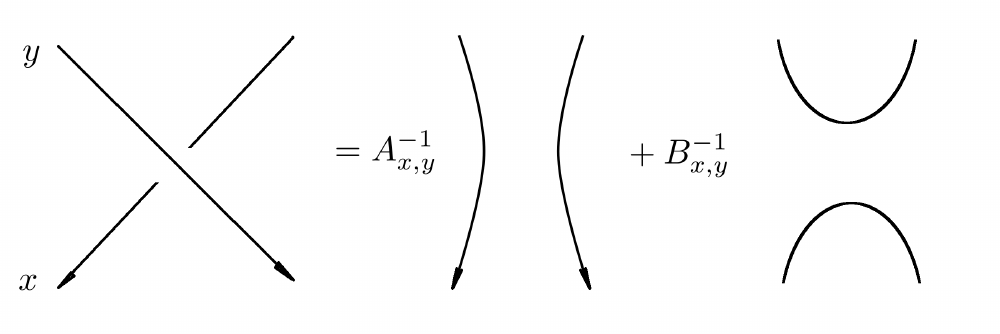}\]
with $\delta$ the value of a simple closed curve and $w$ the value of a 
positive crossing
\[\includegraphics{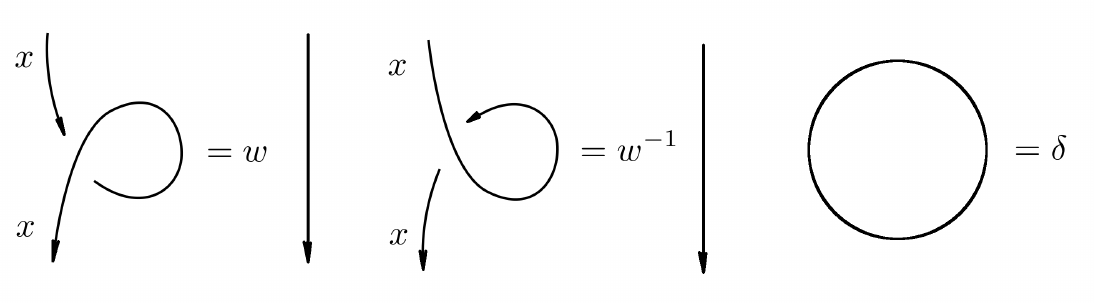}\]
is invariant under $X$-colored Reidemeister moves. More precisely, for each
$X$-coloring $L_f$ of an $n$-crossing diagram $L$:
\begin{itemize}
\item A \textit{state} of $L_f$ is a choice of $C_j=A_{xy}$ or $B_{xy}$ 
smoothing at every crossing $j=1,\dots, n$;
\item For each state we compute the product of the $C_j$s and $\delta$s for 
each component of the smoothed state, and
\item Multiply the sum of these over the
set of states by the \textit{writhe correction factor}
$w^{n-p}$ where $n,p$ are the number of negative and positive crossings
respectively,
\item Obtaining the \textit{state-sum}
\[\beta(L_f)=w^{n-p}\sum_{\mathrm{states}} 
\left(\left(\prod_{\mathrm{smoothings}}C_j\right)\left(\prod_{\mathrm{components}}\delta \right)\right). \]
\end{itemize}
Then the multiset of such state-sum values over the set of 
$X$-colorings is the \textit{biquandle bracket multiset invariant} of $L$
with respect to the biquandle $X$ and bracket $\beta$, denoted
$\Phi_X^{\beta, M}(L).$ It is common practice to convert the multiset invariant
to a ``polynomial'' form by writing the elements of the multiset as exponents
of a dummy variable $u$ with multiplicities as coefficients, e.g. converting
$\{0,0,1,1,1,5\}$ to $2+3u+u^5$.

\begin{example}
Let $X=\{1,2\}$ with $1^1=1^2=1_1=1_2=2$ and $2^1=2^2=2_1=2_2=1$ and 
$R=\mathbb{Z}_7$. Then one verifies that the matrix
\[\left[\begin{array}{rr|rr}
1 & 6 & 2 & 5 \\
4 & 1 & 1 & 2
\end{array}\right]\]
defines a biquandle bracket with $\delta=-1(2)-1(4)=-6=1$ and $w=-1^24=-4=3$. 
The Hopf link below has four $X$-colorings
\[\includegraphics{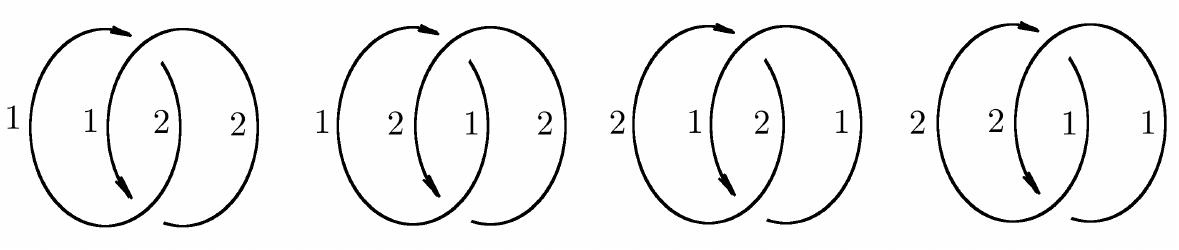}\]
each of which expands to four states with coefficients and state-sums
as pictured. 
\[\includegraphics{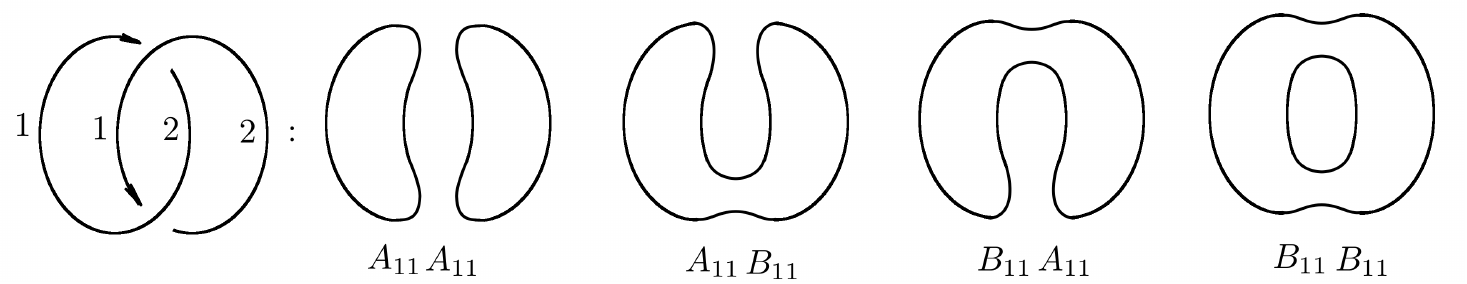}\]
\[\includegraphics{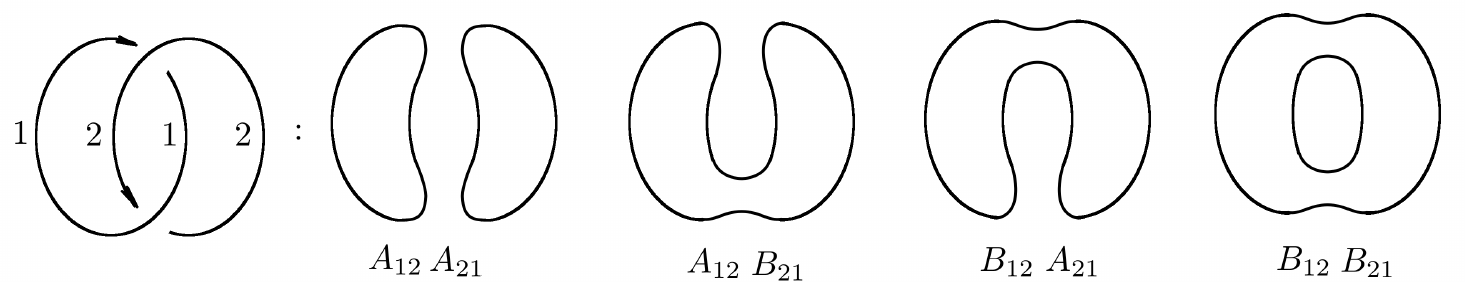}\]
\[\includegraphics{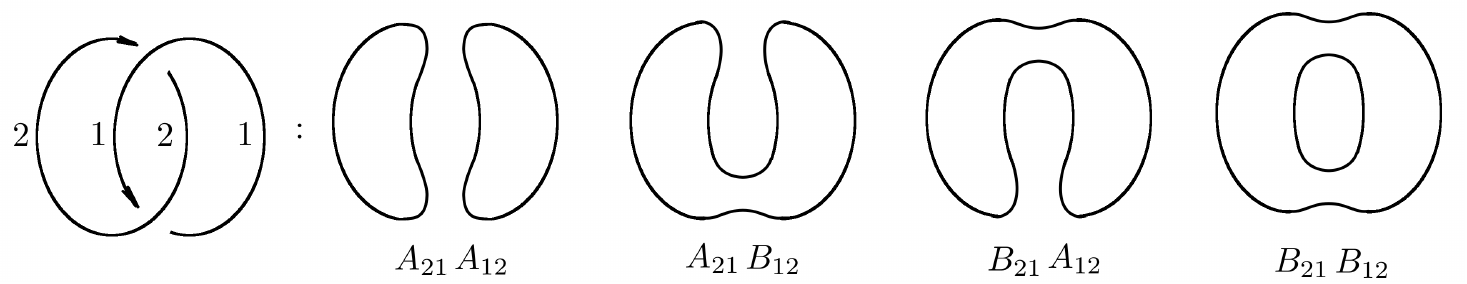}\]
\[\includegraphics{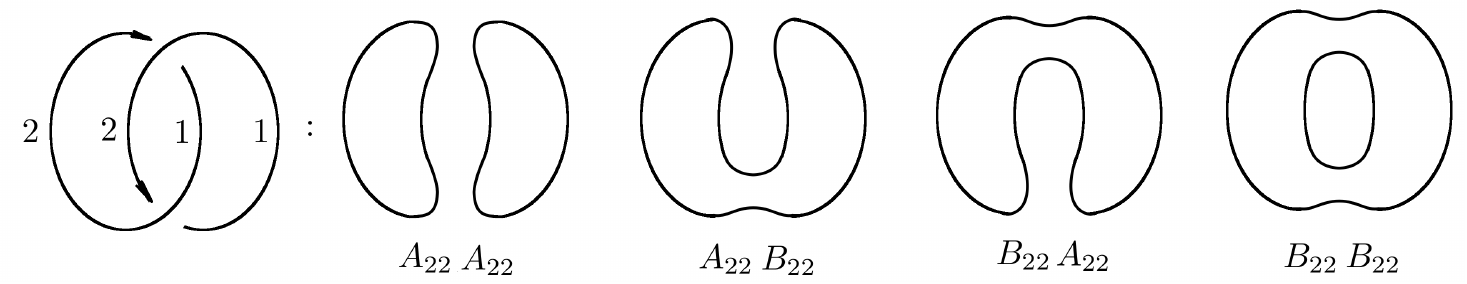}\]
The the state-sums are 
\begin{eqnarray*}
\beta\left(\raisebox{-0.4in}{\includegraphics{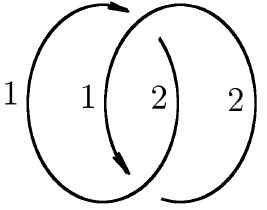}}\right)  
& = &  w^{-2}(\delta^2A_{11}^2+2\delta A_{11}B_{11}+\delta^2B_{11}^2) \\ 
& = &  3^{-2}(1^21^2+(1)(1)(2)+(1)(2)(1)+1^22^2)=1, \\
\beta\left(\raisebox{-0.4in}{\includegraphics{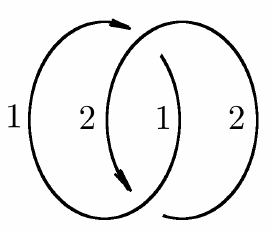}}\right)  
& = &  
w^{-2}(\delta^2A_{12}A_{21}+\delta A_{12}B_{21}+\delta B_{12}A_{21}+\delta^2B_{12}B_{21}) \\ & = &  3^{-2}(1^2(6)(4)+(1)(6)(1)+(1)(5)(4)+1^2(5)(1))=3,\\
\beta\left(\raisebox{-0.4in}{\includegraphics{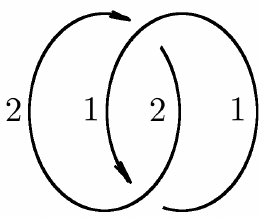}}\right)  
& = &  
w^{-2}(\delta^2A_{21}A_{12}+\delta A_{21}B_{12}+\delta B_{21}A_{12}+\delta^2B_{21}B_{12}) \\ & = & 3^{-2}(1^2(4)(6)+(1)(1)(6)+(1)(4)(5)+1^2(1)(5))=3, \\
\beta\left(\raisebox{-0.4in}{\includegraphics{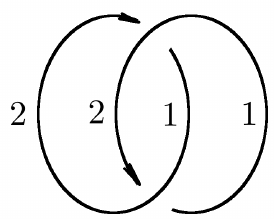}}\right)  
& = &  
w^{-2}(\delta^2A_{22}^+\delta A_{22}B_{22}+\delta B_{22}A_{21}+\delta^2B_{22}^2) 
\\ & = & 3^{-2}(1^2(1)^2+(1)(1)(2)+(1)(2)(1)+1^2(2)^2)=1. 
\end{eqnarray*}
Then the multiset invariant is $\Phi_X^{\beta,M}(L)=\{1,1,3,3\}$ or in 
polynomial form, $\Phi_X^{\beta}(L)=2u+2u^3$.
\end{example}

\section{Trace Diagrams}\label{T}

In this section we introduce a method for computing this invariant
recursively by applying the skein expansion one crossing at a time
as opposed to the state-sum method of performing all smoothings 
simultaneously. To this end we introduce \textit{trace diagrams}.

\begin{definition}
A \textit{trace diagram} is a planar diagram with \textit{crossings}
and \textit{signed traces}:
\begin{itemize}
\item[(i)] A \textit{crossing} is a degree four vertex with pass-through
orientation and crossing information, i.e. pairs of edges resolve into 
an oriented over-crossing strand and an oriented under-crossing strand as
depicted:
\[\includegraphics{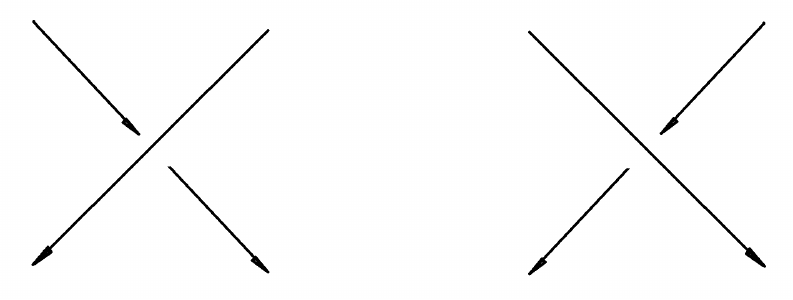}\]
\item[(ii)] The degree three vertices have two oriented edges and one 
unoriented dashed edge called a \textit{trace}, decorated with a $+$ or 
$-$ sign as depicted; we require that each trace either connects two 
parallel oriented pass-through vertices or connects a bivalent 
sink to a bivalent source so that the neighborhood of each trace is as depicted.
We will refer to the former as \textit{type A} traces and the later as 
\textit{type B} traces.
\[\includegraphics{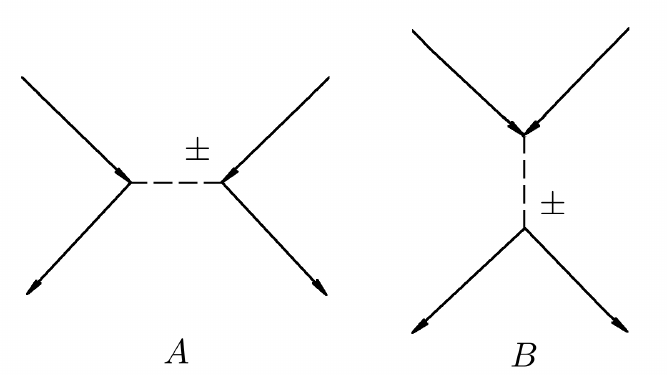}\]
\end{itemize}
\end{definition}

Unsigned trace diagrams have appeared in the literature before, with traces 
recording the sites of smoothings in the Kauffman bracket expansion of a knot 
or link. Our general idea is that a trace with a $+$ or $-$ sign respectively
will act like a positive or negative crossing respectively for the purposes 
of biquandle colorings and Reidemeister moves. More precisely, a 
\textit{biquandle coloring} of a trace diagram by a biquandle $X$ is an 
assignment of elements of $X$ to the directed edges in $X$ such that the 
crossing conditions from section \ref{B} and following conditions at every 
trace are satisfied:

\[\includegraphics{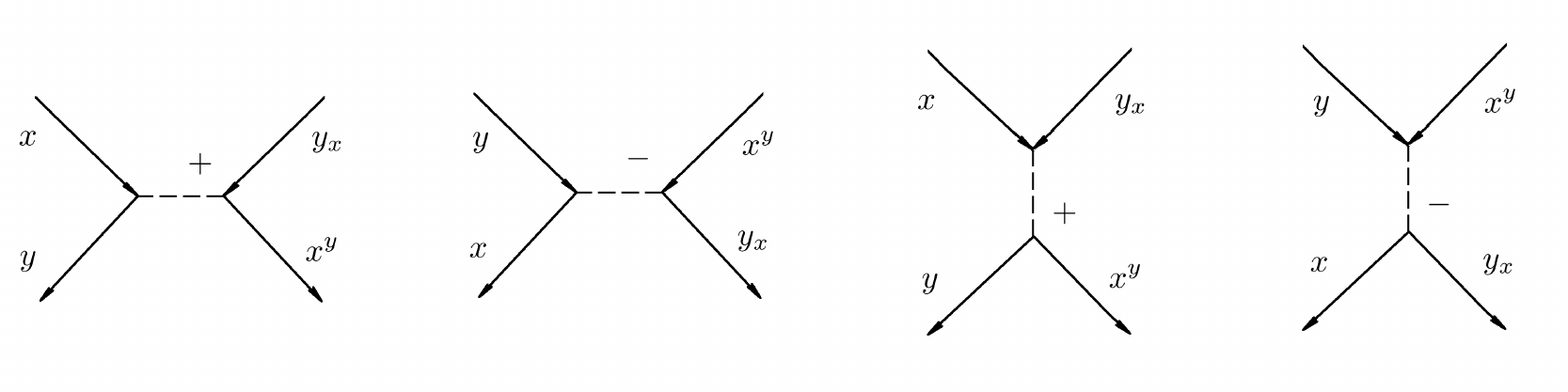}\]

\begin{definition}
Let $X$ be a finite biquandle, $R$ a commutative ring with identity and
$\beta$ an $X$-bracket over $R$. Define a map $[\ ]:\mathcal{L}_X\to R$ from 
the set of $X$-colored oriented trace diagrams $\mathcal{L}_X$ to $R$ 
recursively by the rules
\begin{itemize}
\item[(i)] 
\[\includegraphics{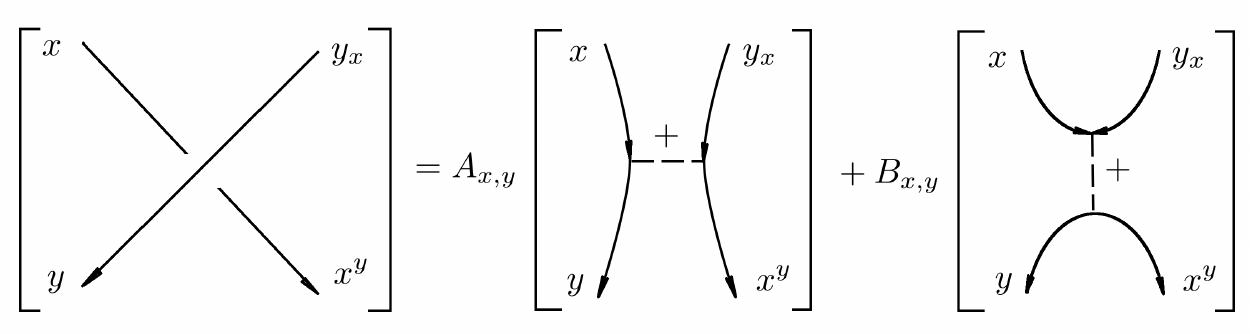}\]
\[\includegraphics{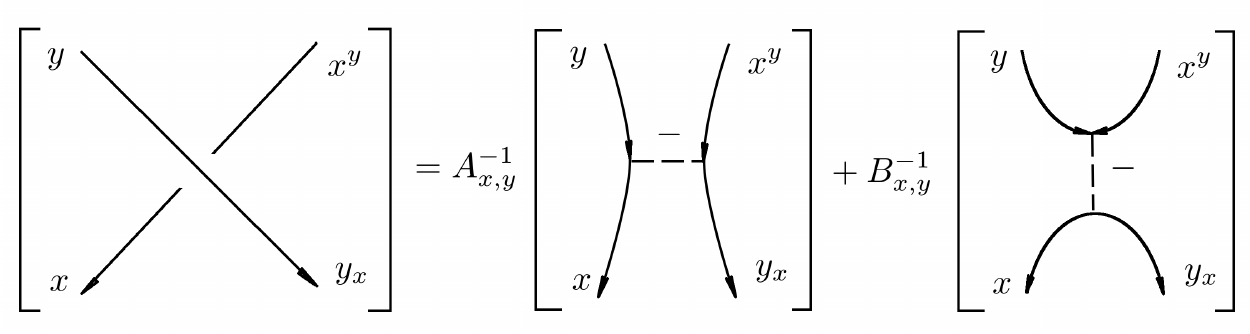}\]
and
\item[(ii)] If $D$ is a trace diagram with no crossings, then
$[D]=w^{n-p}\delta^k$ where $n$ is the number of negative traces, 
$p$ is the number of positive traces and $k$ is the number
of components (i.e., simple closed curves) in the diagram obtained 
by deleting all the of the traces in $D$.
\end{itemize}
\end{definition}

\begin{observation}\label{ob:1}
If $D$ and $D'$ are two $X$-colored trace diagrams which 
\begin{itemize}
\item Are identical outside a neighborhood $N$,
\item Have the same connectivity on the boundary of $N$, i.e., points of 
$\partial N$ which are connected by strands inside $D$ after traces are deleted
are connected inside $D'$ after traces are deleted and points of $\partial N$ 
which are not connected inside $D$ after traces are deleted are not connected 
inside $D'$ after traces are deleted, and
\item Have equal contributions of coefficients, $\delta$s and $w$s,
\end{itemize} 
then $[D]=[D']$.
\end{observation}

It then follows easily that:

\begin{proposition}
For any finite biquandle $X$ and $X$-colored oriented link diagram $L_f$, the
value $[L_f]\in R$ is unchanged by Reidemeister moves. In particular, we have 
$[L_f]=\beta(L_f)$.
\end{proposition}

%

Thus, trace diagrams give us a way of doing skein expansion of biquandle 
bracket invariants recursively like we do with classical skein invariants
in a way that preserves the biquandle colorings. However, this is somewhat
unsatisfying since the main advantage of the recursive skein expansion
is the ability to smooth, then apply Reidemeister moves to simplify the diagram
before doing additional smoothings. While we \textit{can} perform 
Reidemeister moves on trace diagrams, so far we have only allowed genuine
Reidemeister moves not involving traces, i.e., moves we could have already 
performed prior to smoothing. Hence we ask, what happens when we move a strand 
of a trace diagram past a trace?

There are 16 possible oriented trace moves, each of which imposes one of two
possible sets of conditions on the biquandle bracket coefficients. 
We note that for each move involving a trace of type A
there is a corresponding move obtained by replacing the type A trace with 
a type B trace. 
\[\includegraphics{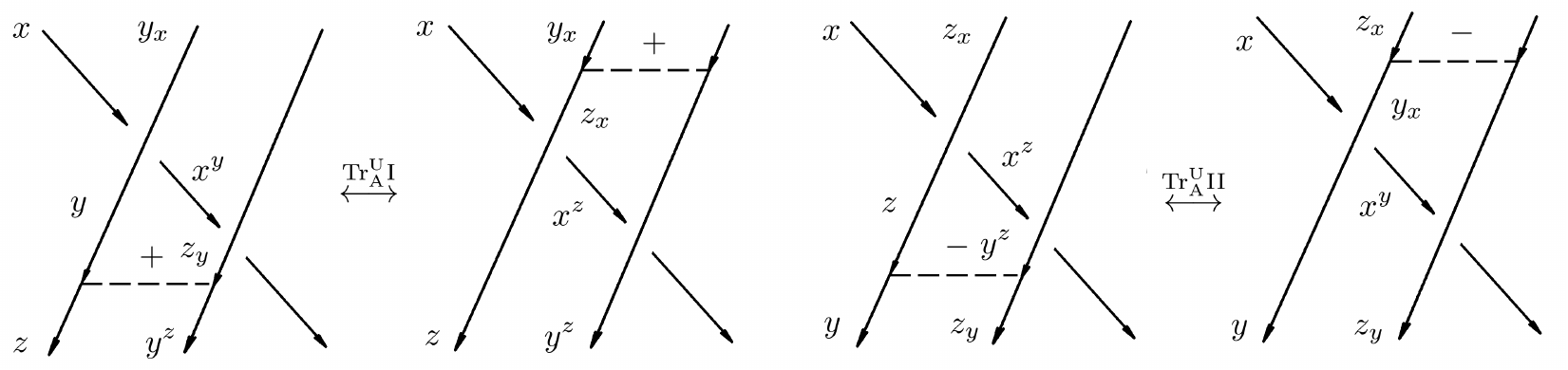}\]
\[\includegraphics{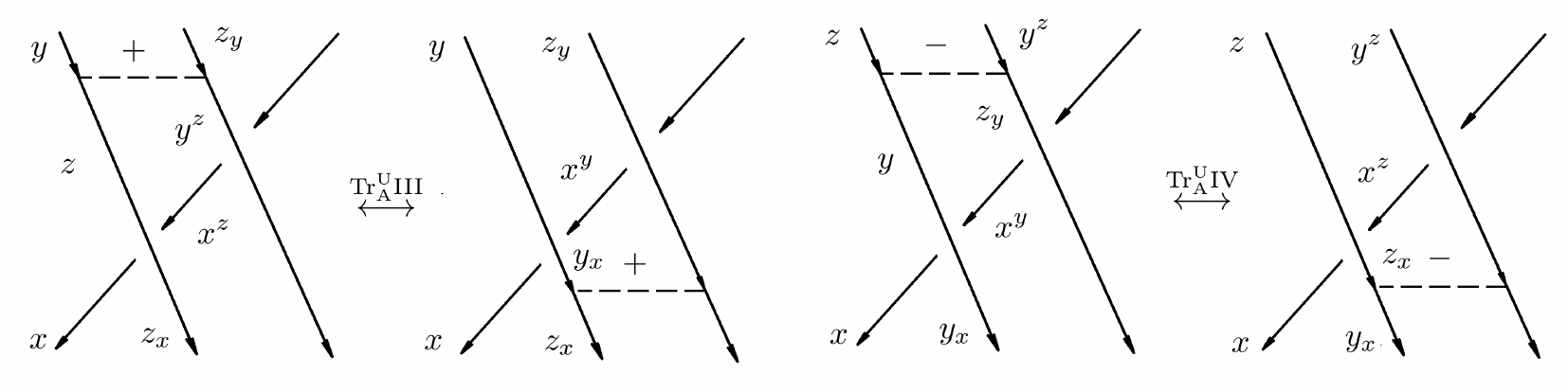}\]
\[\includegraphics{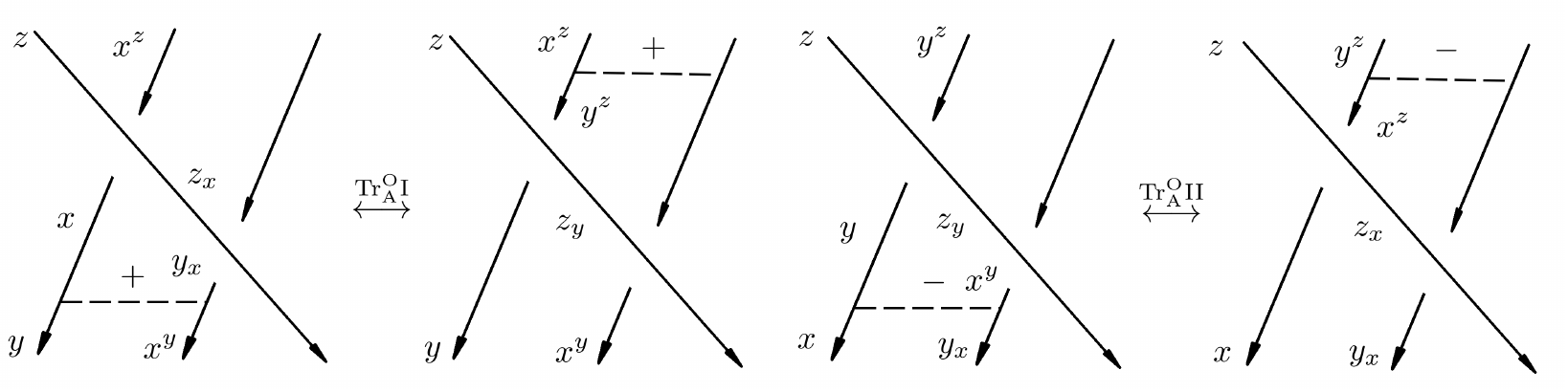}\]
\[\includegraphics{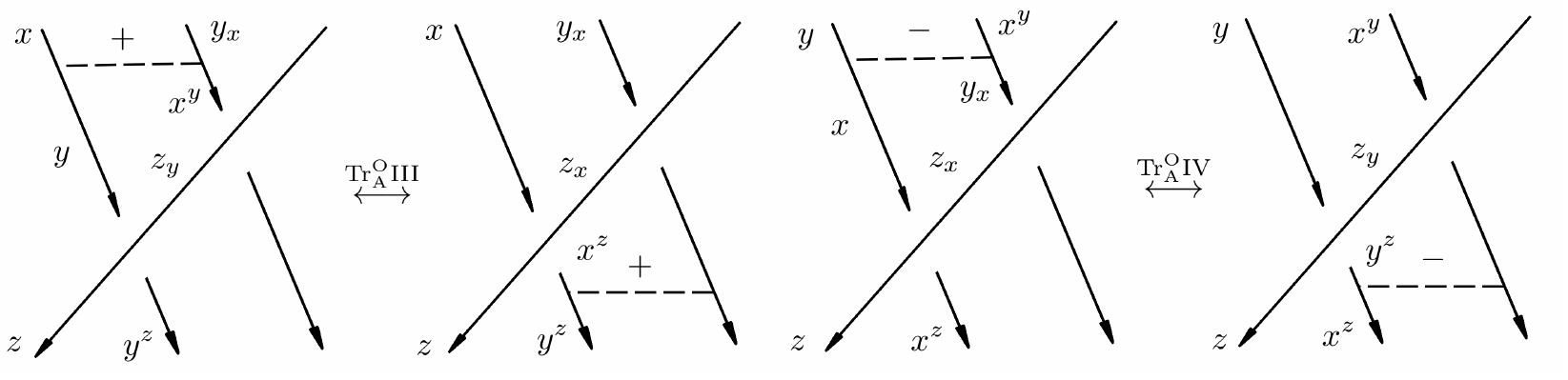}\]
\[\includegraphics{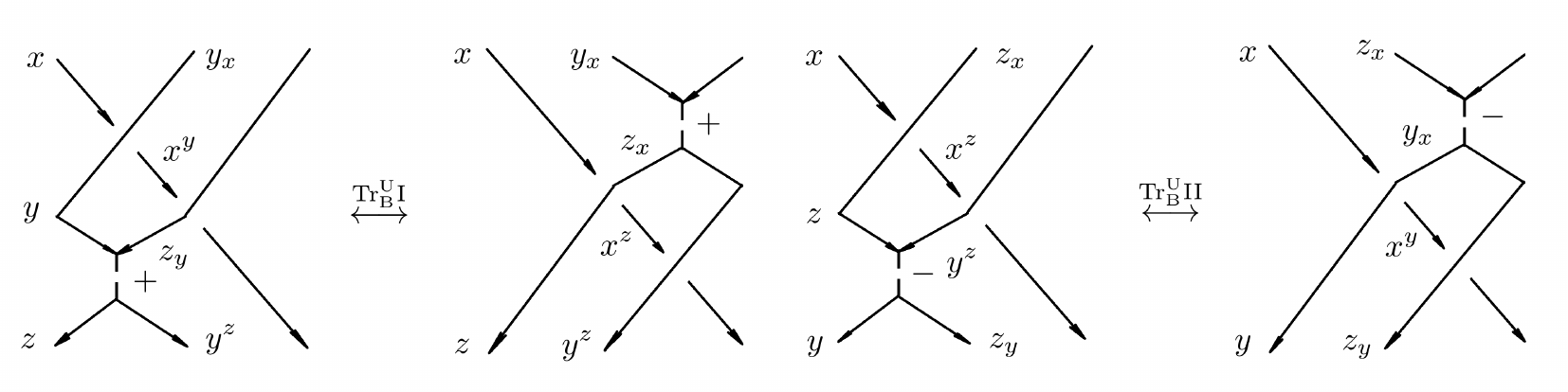}\]
\[\includegraphics{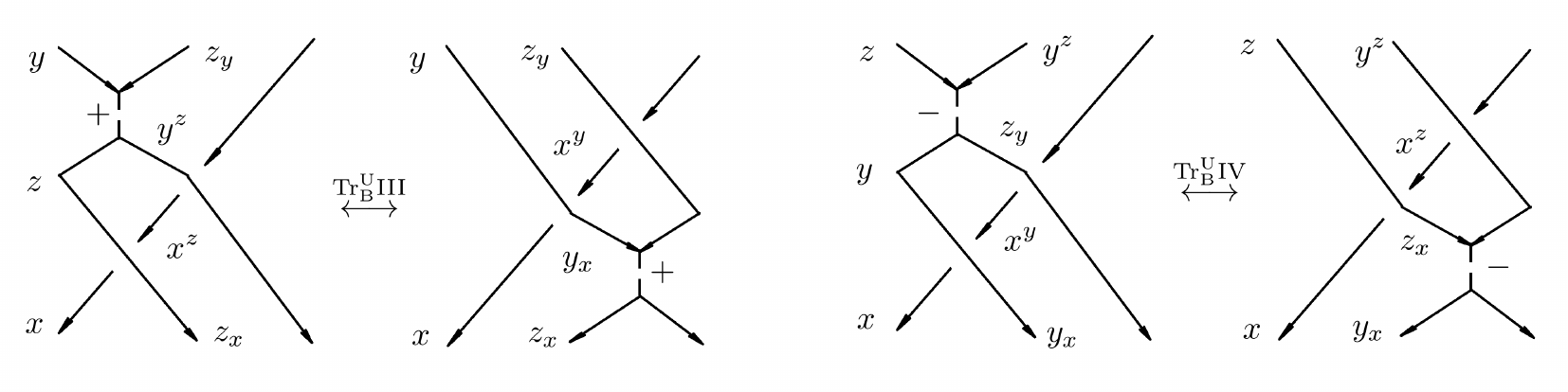}\]
\[\includegraphics{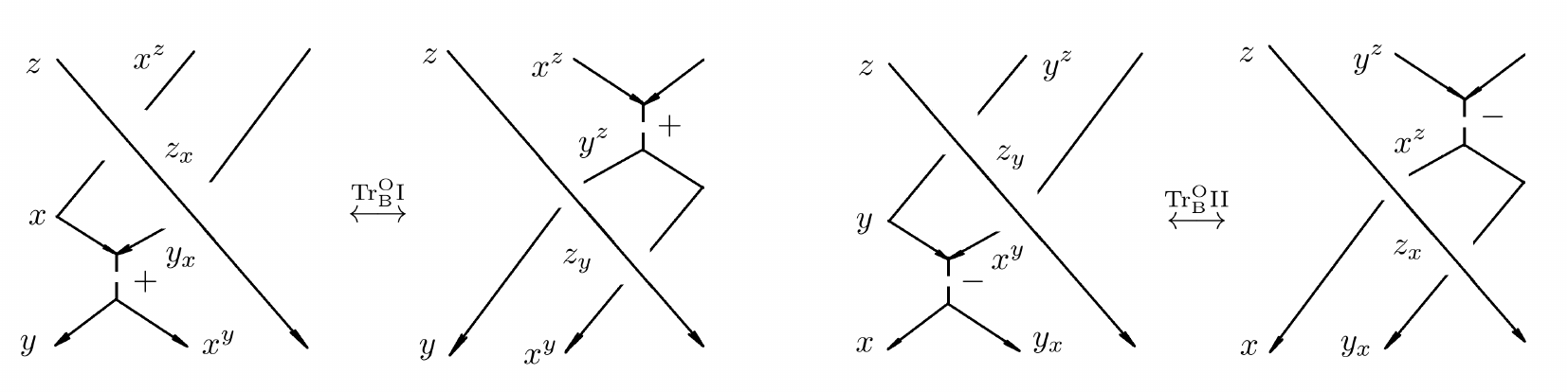}\]
\[\includegraphics{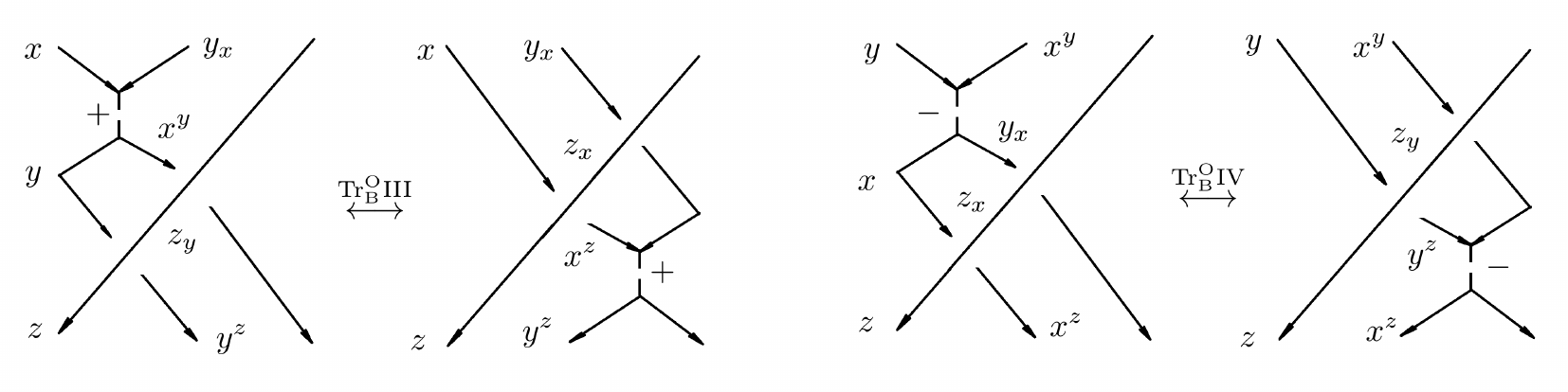}\]


\begin{theorem}
The necessary and sufficient conditions for moving a strand over a trace of 
type $A$ are the same as those for moving a strand over a trace of type $B$,
namely with colors $(x,y,z)$ as depicted above, the coefficients must 
satisfy the conditions
\[A_{x,y}=A_{x^z,y^z}\quad \mathrm{and}\quad
A_{y,z}B_{x^y,z_y} = B_{x,z}A_{y_x,z_x}=A_{x,z}B_{y_x,z_x}=B_{y,z}A_{x^y,z_y}.\]
The necessary and sufficient conditions for moving a strand under a trace of 
type $A$ are the same as those for moving a strand under a trace of type $B$, 
namely with colors $(x,y,z)$ as depicted above, the coefficients must 
the coefficients satisfy the conditions
\[A_{y,z}=A_{y_x,z_x}\quad \mathrm{and}\quad
A_{x,y}B_{x^y,z_y} = B_{x,z}A_{x^z,y^z}=A_{x,z}B_{x^z,y^z}=B_{x,y}A_{x^y,z_y}.\]
\end{theorem}

\begin{proof} We consider the case of undercrossing moves; the overcrossing
case is similar.
Consider the move $Tr_A^UI$ with the strands colored as depicted.
\[\includegraphics{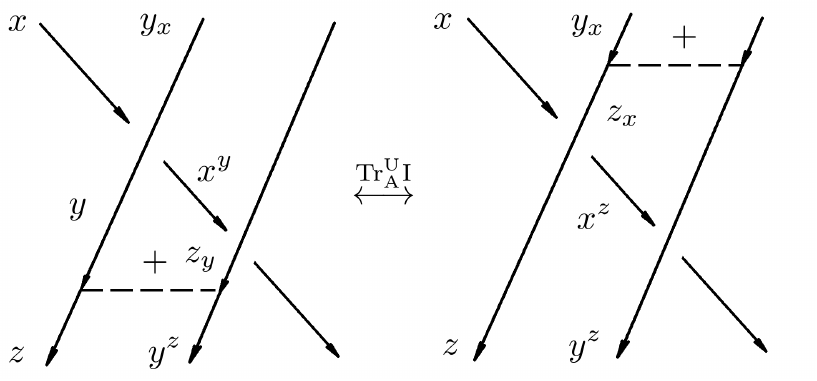}\]
The two sides expand to 
\[\scalebox{0.95}{\includegraphics{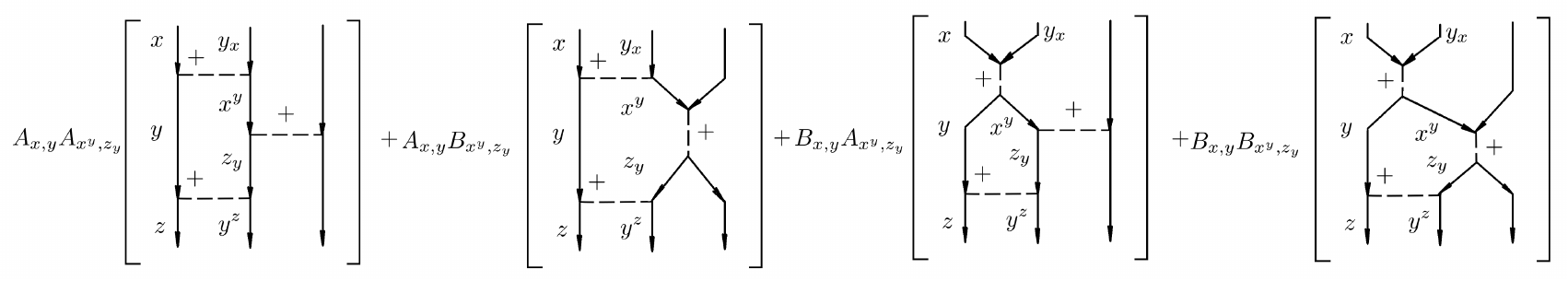}}\]
and
\[\scalebox{0.95}{\includegraphics{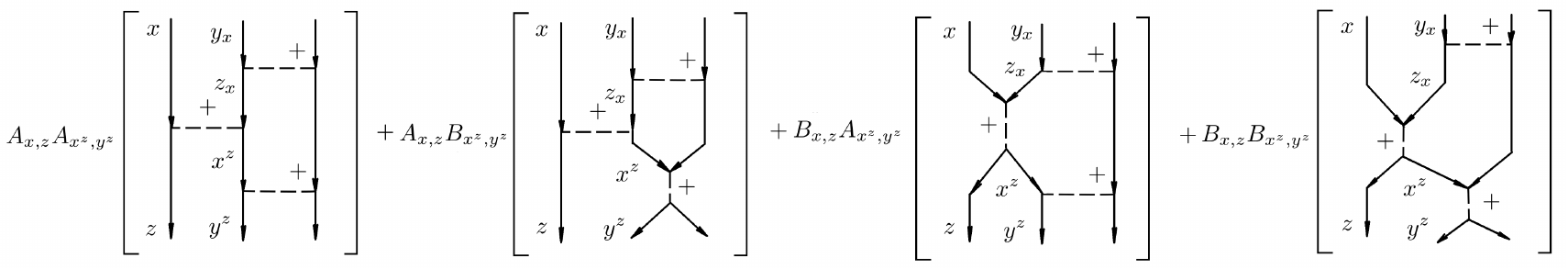}}.\]
Comparing coefficients of diagrams with the same boundary connectivity and
writhe information, we obtain the requirements that for all $x,y,z\in X$ we have
\[\begin{array}{rcll}
A_{x,y}A_{x^y,z_y} & = & A_{x,z}A_{x^z,y^z} & (i)\\
A_{x,y}B_{x^y,z_y} & = & A_{x,z}B_{x^z,y^z} & (ii)\\
B_{x,y}A_{x^y,z_y} & = & B_{x,z}A_{x^z,y^z} & (iii)\ \ \mathrm{and}\\
B_{x,y}B_{x^y,z_y} & = & B_{x,z}B_{x^z,y^z} & (iv).
\end{array}\]
Suppose we have a biquandle bracket satisfying $(i)$ through $(iv)$.
Equations (i) and (iv) are equivalent via the biquandle 
bracket axioms $A_{x,y}A_{y,z}A_{x^y,z_y} = A_{x,z}A_{y_x,z_x}A_{x^z,y^z}$
and $B_{x,y}A_{y,z}B_{x^y,z_y} = B_{x,z}A_{y_x,z_x}B_{x^z,y^z}$
to $A_{y,z}=A_{y_x,z_x}$.
To see that the left sides of equations (ii) and (iii) represent
equal elements of $R$, consider the biquandle bracket axiom 
\[A_{x,y}A_{y,z}B_{x^y,z_y} =A_{x,z}B_{y_x,z_x}A_{x^z,y^z} +A_{x,z}A_{y_x,z_x}B_{x^z,y^z}   +\delta A_{x,z}B_{y_x,z_x}B_{x^z,y^z} +B_{x,z}B_{y_x,z_x}B_{x^z,y^z}.\] 
Since $A_{y,z}=A_{y_x,z_x}$ and 
$A_{x,y}B_{x^y,z_y} = A_{x,z}B_{x^z,y^z}$, we have
$A_{x,y}A_{y,z}B_{x^y,z_y} =A_{x,z}A_{y_x,z_x}B_{x^z,y^z}$
and our equation reduces to
\[-\delta A_{x,z}B_{y_x,z_x}B_{x^z,y^z}  
=A_{x,z}B_{y_x,z_x}A_{x^z,y^z} +B_{x,z}B_{y_x,z_x}B_{x^z,y^z}.\] 
Then multiplying through by $B_{y_x,z_x}^{-1}$ we have
\[-\delta A_{x,z}B_{x^z,y^z}  
=A_{x,z}A_{x^z,y^z} +B_{x,z}B_{x^z,y^z}.\] 
Next, noting that $\delta=-A_{x,z}^{-1}B_{x,z}-A_{x,z}B_{x,z}^{-1}$, we have
\[B_{x,z}B_{x^z,y^z}  +B_{x,z}^{-1}A_{x,z}^2B_{x^z,y^z}=A_{x,z}A_{x^z,y^z} +B_{x,z}B_{x^z,y^z}\] 
which implies
\[B_{x,z}^{-1}A_{x,z}^2B_{x^z,y^z} = A_{x,z}A_{x^z,y^z}\]
whence
\[A_{x,z}B_{x^z,y^z} =  B_{x,z}A_{x^z,y^z}.\]
The other undercrossing type A moves yield the same equations; hence, the
undercrossing moves require and are satisfied by the conditions
\[A_{y,z}=A_{y_x,z_x}\] and
\[A_{x,y}B_{x^y,z_y} = B_{x,z}A_{x^z,y^z}=A_{x,z}B_{x^z,y^z}=B_{x,y}A_{x^y,z_y}.\]

For undercrossing type B moves, expanding the
two sides of the move $\mathrm{Tr}^U_BI$ as shown,
\[\includegraphics{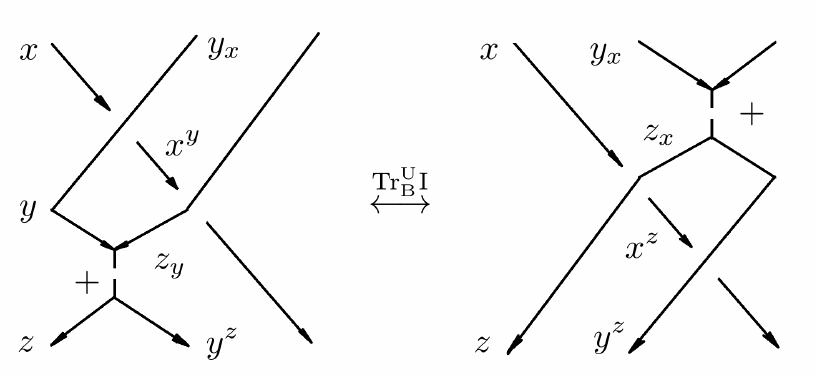}\]
we obtain
\[\scalebox{0.95}{\includegraphics{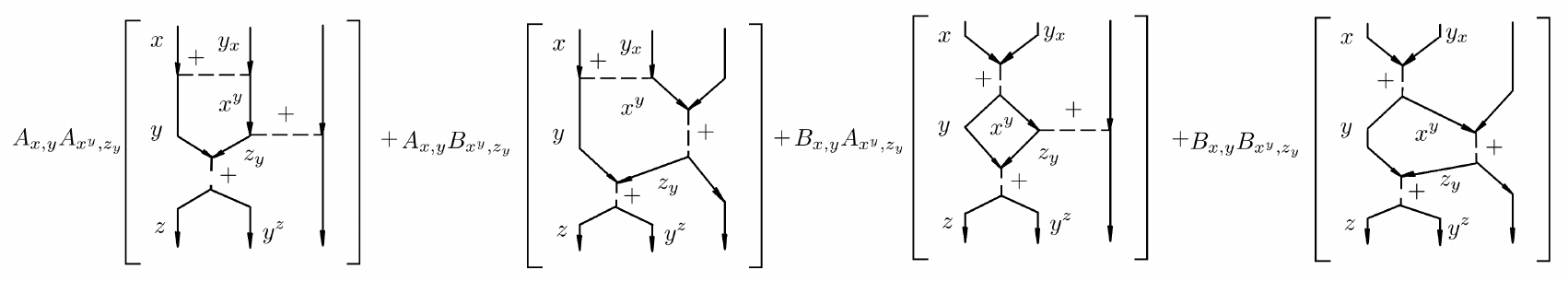}}\]
and
\[\scalebox{0.95}{\includegraphics{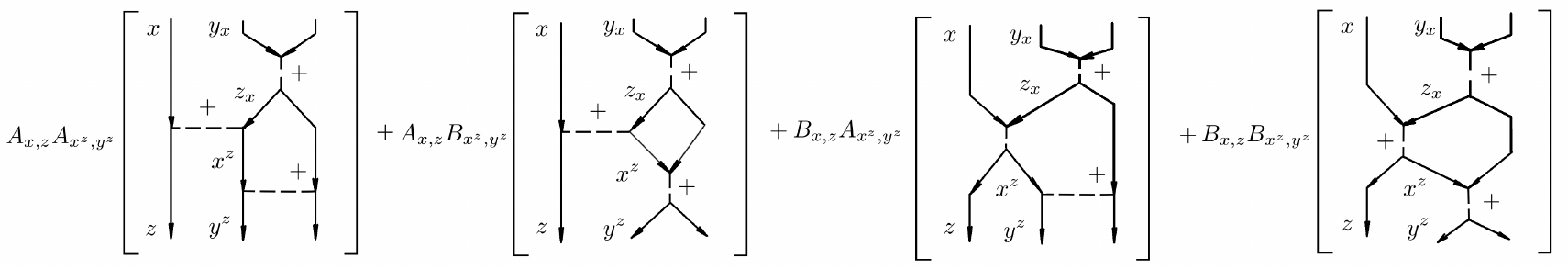}}.\]
Comparing coefficients, we obtain the requirements that
\begin{eqnarray*}
A_{x,y}B_{x^y,z_y}& = & B_{x,z}A_{z^x,y^x}\\
0 & = & A_{x,y}A_{x^y,z_y} +\delta B_{x,y}A_{x^y,z_y} + B_{x,y}B_{x^y,z_y}\\
0 & = & A_{x,z}A_{x^z,y^z} +\delta A_{x,z}B_{x^z,y^z} + B_{x,z}B_{x^z,y^z}
\end{eqnarray*}
Recalling that $\delta=-A_{x,y}B_{x,y}^{-1}-A_{x,y}^{-1}B_{x,y}$, the second
equation is equivalent to
\begin{eqnarray*}
(A_{x,y}B_{x,y}^{-1}+A_{x,y}^{-1}B_{x,y}) B_{x,y}A_{x^y,z_y}  & = & A_{x,y}A_{x^y,z_y} + B_{x,y}B_{x^y,z_y}\\
A_{x,y}A_{x^y,z_y}+A_{x,y}^{-1}B_{x,y}^2A_{x^y,z_y}  & = & A_{x,y}A_{x^y,z_y} + B_{x,y}B_{x^y,z_y}\\
A_{x,y}^{-1}B_{x,y}^2A_{x^y,z_y}  & = & B_{x,y}B_{x^y,z_y}\\
B_{x,y}A_{x^y,z_y}  & = & A_{x,y}B_{x^y,z_y}
\end{eqnarray*}
and similarly the third equation is equivalent to
$A_{x,z}B_{x^z,y^z}=B_{x,z}A_{x^z,y^z}$.
Then move $Tr_B^UI$ requires and is satisfied by the conditions
\[A_{x,y}B_{x^y,z_y} = B_{x,z}A_{x^z,y^z}=A_{x,z}B_{x^z,y^z}=B_{x,y}A_{x^y,z_y}.\]
Then $A_{x,z}B_{x^z,y^z}=B_{x,z}A_{x^z,y^z}$ implies 
$B_{x,z}^{-1}A_{x,z}^2B_{x^z,y^z} = A_{x,z}A_{x^z,y^z}$ which implies
\[B_{x,z}B_{x^z,y^z}  +B_{x,z}^{-1}A_{x,z}^2B_{x^z,y^z}=A_{x,z}A_{x^z,y^z} +B_{x,z}B_{x^z,y^z}\]
and then
\[0=A_{x,z}A_{x^z,y^z} +\delta A_{x,z}B_{x^z,y^z} +B_{x,z}B_{x^z,y^z}\]
so
\[0 =A_{x,z}A_{y_x,z_x}B_{x^z,y^z}   +\delta A_{x,z}B_{y_x,z_x}B_{x^z,y^z} +B_{x,z}B_{y_x,z_x}B_{x^z,y^z}.\] 
Comparing with the biquandle bracket axiom
\[A_{x,y}A_{y,z}B_{x^y,z_y} =A_{x,z}B_{y_x,z_x}A_{x^z,y^z} +A_{x,z}A_{y_x,z_x}B_{x^z,y^z}   +\delta A_{x,z}B_{y_x,z_x}B_{x^z,y^z} +B_{x,z}B_{y_x,z_x}B_{x^z,y^z},\] 
we obtain $A_{x,y}A_{y,z}B_{x^y,z_y} =A_{x,z}A_{y_x,z_x}B_{x^z,y^z}$
and since $A_{x,y}B_{x^y,z_y} = A_{x,z}B_{x^z,y^z}$, we obtain
$A_{y,z}=A_{y_x,z_x}$ and the type $B$ undercrossing conditions imply the 
type $A$ undercrossing conditions.
\end{proof}

\begin{definition}
A biquandle bracket $\beta$ over a biquandle $X$ is \textit{adequate}
for a trace move if it satisfies the algebraic conditions associated 
with the move. A biquandle bracket is 
\begin{itemize}
\item \textit{Over-Adequate} if it is adequate for all overcrossing
trace moves, i.e., if for all $x,y,z\in X$ we have
\[A_{y,z}=A_{y^x,z^x}\quad \mathrm{and}\quad
A_{y,z}B_{x^y,z_y} = B_{x,z}A_{y_x,z_x}=A_{x,z}B_{y_x,z_x}=B_{y,z}A_{x^y,z_y},\]
\item \textit{Under-Adequate} if it is adequate for all undercrossing
trace moves, i.e., if for all $x,y,z\in X$ we have 
\[A_{y,z}=A_{y_x,z_x}\quad \mathrm{and}\quad
A_{x,y}B_{x^y,z_y} = B_{x,z}A_{x^z,y^z}=A_{x,z}B_{x^z,y^z}=B_{x,y}A_{x^y,z_y},\]
and
\item \textit{Adequate} if it is both over- and under-adequate.
\end{itemize}
\end{definition}


\begin{example}
Constant brackets where $A_{x,y}=A$ and $B_{x,y}=B$ such as those defining
the classical skein invariants are adequate.
\end{example}

If a biquandle bracket is over-adequate, then we can freely move strands  
over traces during the skein expansion
without changing the biquandle 
bracket value. If a bracket is under adequate, we can move strands under 
traces. If a bracket is adequate, we can move strands both over and under,
matching the uncolored skein expansion case.

\begin{example} Biquandle brackets come in all four global types.
Consider the biquandle $X$ defined by the operation matrix
\[\left[\begin{array}{rrr|rrr}
3 & 1 & 3 & 3 & 3 & 3 \\
2 & 2 & 2 & 2 & 2 & 2 \\
1 & 3 & 1 & 1 & 1 & 1
\end{array}\right]\] 
Our Python computations indicate that the brackets over $\mathbb{Z}_5$ below
are adequate, over-adequate, under-adequate, and neither respectively:
\[\begin{array}{cccc}
\left[\begin{array}{rrr|rrr}
1 & 1 & 1 & 2 & 2 & 2 \\
2 & 1 & 2 & 4 & 2 & 4 \\
1 & 1 & 1 & 2 & 2 & 2
\end{array}\right] &
\left[\begin{array}{rrr|rrr}
1 & 3 & 1 & 2 & 4 & 2 \\
1 & 4 & 1 & 2 & 2 & 2 \\
1 & 3 & 1 & 2 & 4 & 2
\end{array}\right] &
\left[\begin{array}{rrr|rrr}
1 & 2 & 1 & 3 & 1 & 3 \\
2 & 4 & 2 & 4 & 3 & 4 \\
1 & 2 & 1 & 3 & 1 & 3
\end{array}\right] &
\left[\begin{array}{rrr|rrr}
1 & 2 & 4 & 2 & 1 & 3 \\
1 & 4 & 4 & 2 & 2 & 3 \\
4 & 2 & 1 & 3 & 1 & 2
\end{array}\right] \\
\mathrm{Adequate} &
\mathrm{Over-Adequate} &
\mathrm{Under-Adequate} &
\mathrm{Neither} 
\end{array}\]
\end{example}

Of what use are trace diagrams? Coupled with the following proposition, they
can be used to simplify the computation of biquandle bracket invariants analogously
to how classical skein invariants can be computed by applying the skein relation
to reduce to a linear combination of brackets of unlinks.
First, we need a few definitions.

\begin{definition}
A crossing in a bivalent spatial graph diagram is \textit{single-component} 
if both its 
over-crossing and under-crossing strands are in the same component of the link;
a crossing is \textit{multi-component} if its over-crossing and under-crossing 
strands lie on different components of the link.
\end{definition}

Next we define the \textit{magnetic parity} of a crossing in a trace diagram
by thinking analogusly with magnetic graphs like those in \cite{KM}:

\begin{definition}
Let $D$ be a trace diagram and $c$ a single-component crossing of the
diagram obtained by deleting traces. The \textit{magnetic parity}
of $c$ is the parity (even or odd) of the number of orientation-reversing 
vertices between the over- and under-crossing points of $x$ on the component of
$D$ containing $c$ obtained by deleting traces.
\end{definition}

Considering the orientations of strands yields the following:

\begin{lemma}
The magnetic parity of a single-component crossing determines how it closes
as shown:
\[\includegraphics{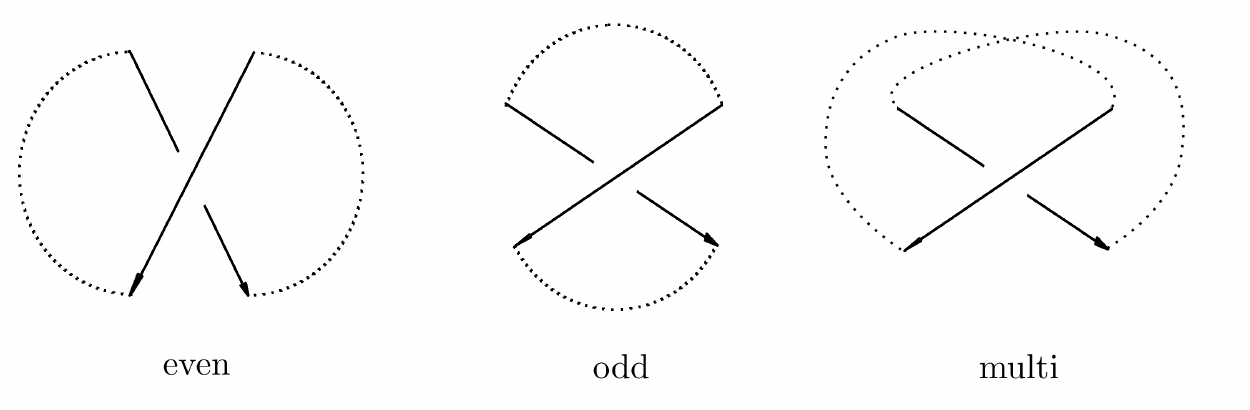}\]
A multicomponent crossing must have even magnetic parity.
\end{lemma}

We observe that trace moves and Reidemeister moves do not change the magnetic 
parity of a crossing, and that magnetic parity is well-defined for 
single-component crossings.

\begin{proposition}\label{ns}
Let $D$ be a trace diagram such that deleting the traces results in an unlink
$U$ which can be reduced to a zero-crossing diagram by Reidemeister I moves.
At each crossing $c$ define a weight $\phi(c)$ according to the following table:
\begin{center}
\begin{tabular}{|c|c|c|c|c|c|}\hline
Crossing & Parity & $\phi(x)$  &Crossing & Parity & $\phi(x)$ \\ \hline
\raisebox{-0.375in}{
\includegraphics{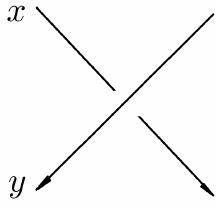}} & Odd & $A_{x,y}+\delta B_{x,y}$ &
\raisebox{-0.375in}{
\includegraphics{sn-no1-46.pdf}} & Even & $\delta A_{x,y}+B_{x,y}$ \\ \hline
\raisebox{-0.375in}{
\includegraphics{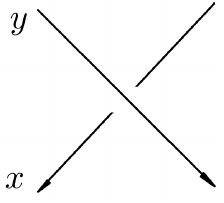}} & Odd & $A_{x,y}^{-1}+\delta B_{x,y}^{-1}$ &
\raisebox{-0.375in}{
\includegraphics{sn-no1-47.pdf}} & Even & $\delta A_{x,y}^{-1}+B_{x,y}^{-1}$ \\ \hline
\end{tabular}\end{center}
Then $[D]$ is given by 
\[[D]=\delta^kw^{n-p}\prod \phi(c) \]
where the product runs over the crossings $c$ of $D$, $k$ is the
number of components of $U$ and $n-p$ is the number of
negative signed traces and crossings minus the number of positive signed 
traces and crossings in $D$.
\end{proposition}

\begin{proof}
Consider the case of a positive crossing.
Without loss of generality we can expand innermost crossings first, i.e. select
a crossing with no other crossings along one arc between its over and under 
instances. Then if this arc contains an odd number of orientation reversals, 
we have
\[\includegraphics{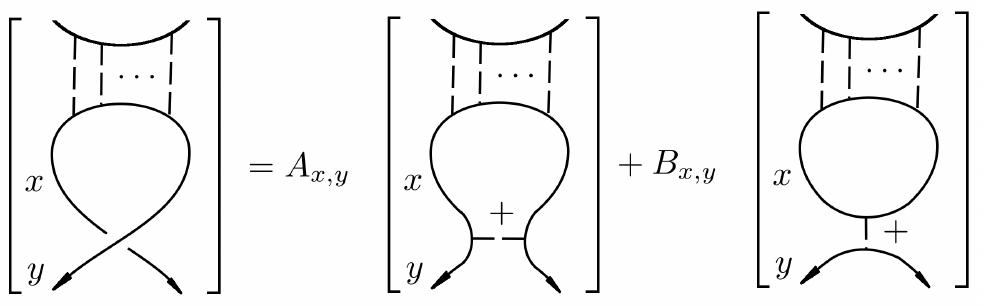}\]
yielding a coefficient contribution of $A_{x,y}+\delta B_{x,y}$
and if it has an even number, we have
\[\includegraphics{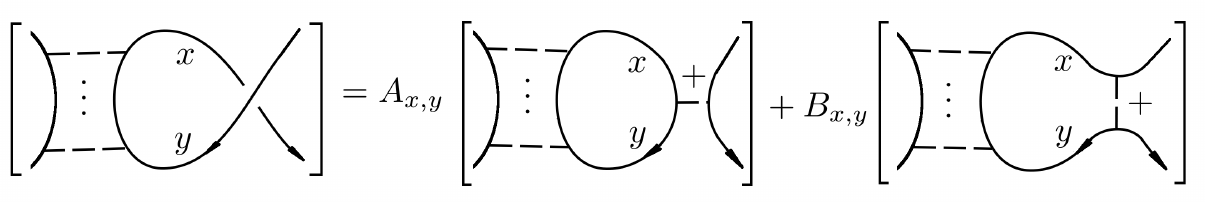}\]
yielding a coefficient contribution of $\delta A_{x,y}+B_{x,y}$.
The negative crossing case is analogous.
\end{proof}

\begin{example}
Consider the trace diagram
\[\includegraphics{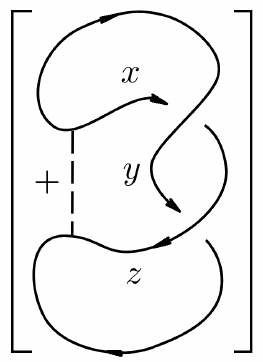}\]
The crossings $c_1$ and $c_2$ both have odd magnetic parity and are both positive 
crossings, so we have $\phi(c_1)=A_{x,y}+\delta B_{x,y}$ and
$\phi(c_1)=A_{y,z}+\delta B_{y,z}$, and we have $C=1$, $n=0$ and $p=3$.
Then expanding via the state sum, we have
\begin{eqnarray*}
\raisebox{-0.4 in}{\scalebox{0.7}{
\includegraphics{sn-no1-49.pdf}}} &= &
A_{x,y}A_{y,z}\raisebox{-0.4 in}{\scalebox{0.7}{\includegraphics{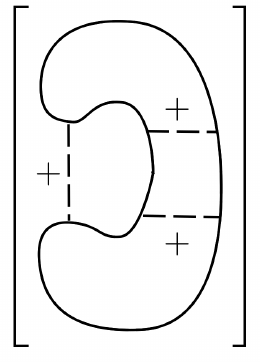}}}
+A_{x,y}B_{y,z}\raisebox{-0.4 in}{\scalebox{0.7}{\includegraphics{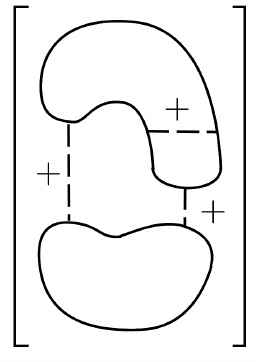}}}
+B_{x,y}A_{y,z}\raisebox{-0.4 in}{\scalebox{0.7}{\includegraphics{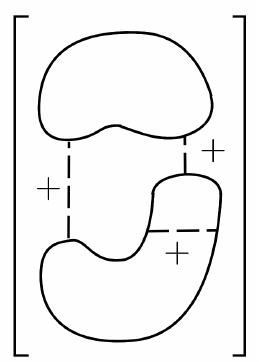}}}
+B_{x,y}B_{y,z}\raisebox{-0.4 in}{\scalebox{0.7}{\includegraphics{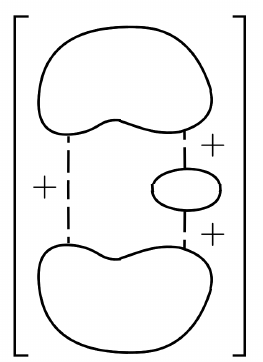}}}\\
& = & 
A_{x,y}A_{y,z}\delta w^{-3}+
A_{x,y}B_{y,z}\delta^2 w^{-3}+
B_{x,y}A_{y,z}\delta^2 w^{-3}+
B_{x,y}B_{y,z}\delta^3 w^{-3}\\
& = & 
A_{x,y}(A_{y,z}+B_{y,z}\delta)\delta w^{-3}+
B_{x,y}(A_{y,z}+B_{y,z}\delta)\delta^2 w^{-3}\\
& = & (A_{x,y}+\delta B_{x,y})(A_{y,z}+B_{y,z}\delta)\delta w^{-3}\\
& = & \phi(c_1)\phi(c_2)\delta^1w^{-3}\\
\end{eqnarray*}
as expected.
\end{example}

\section{\Large\textbf{Monochromatic Moves}}\label{M}

In this section we describe a few moves on trace diagrams in which 
the input colors for the move are all the same. First, we have a 
Homflypt-style skein relation at monochromatic crossings.

Recall that the skein relation for the Homflypt polynomial relates
the invariant of an oriented knot or link with a specified positive
crossing to those of the same knot with the positive crossing replaced with
a negative crossing and with an oriented smoothing. Suppose we have 
a \textit{monochromatic} crossing, i.e.  a crossing in which the two 
left-hand biquandle colors are equal -- biquandle axiom (i) then implies that 
the right-hand colors are equal, and after smoothing, the coloring makes sense
for the oriented smoothing even without the trace. The following lemma 
follows easily from observation \ref{ob:1}. 

\begin{lemma}\label{lem:1}
Let $X$ be a biquandle, $R$ a commutative ring with identity and $\beta$ an 
$X$-bracket over $R$. Then at monochromatic crossings we have the following
identities:
\[\includegraphics{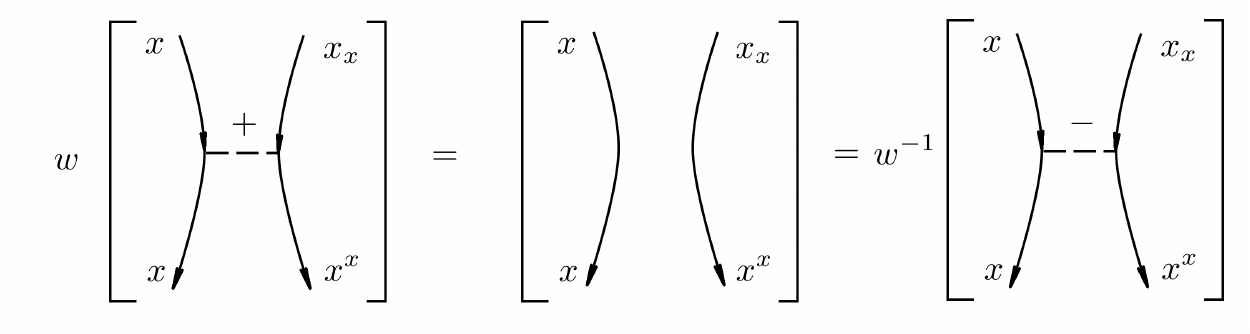}\]and
\[\includegraphics{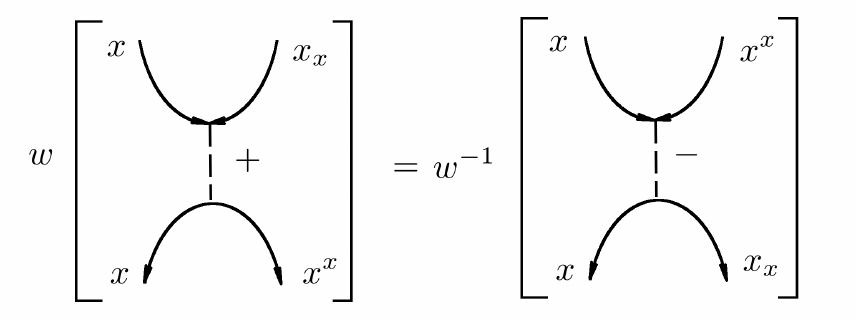}\]
\end{lemma}

Combining the skein relations at a positive and negative crossing, we obtain
\begin{proposition}\label{homfly}
Let $X$ be a biquandle, $\beta$ an $X$-bracket and $L_f$ be an $X$-colored
oriented link diagram. Then at any monochromatic crossing, $[\ ]$ satisfies
the following Homflypt-style skein relation:
\[\scalebox{0.85}{\includegraphics{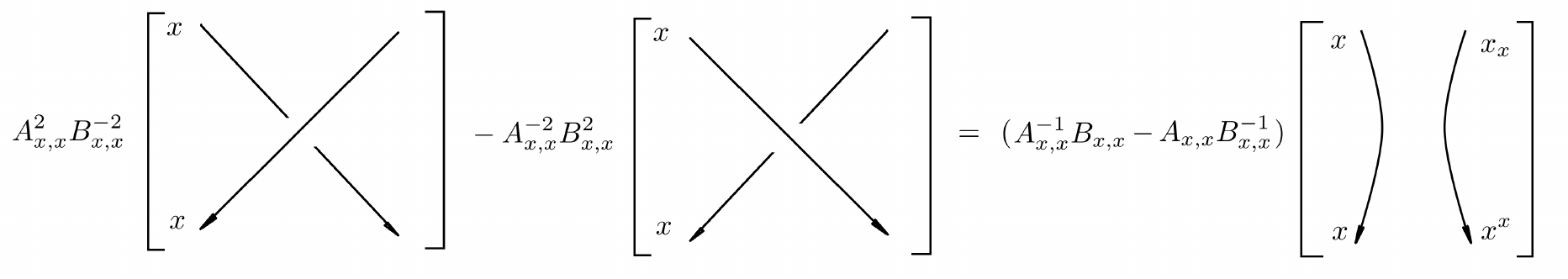}}\]
or equivalently
\[\scalebox{0.85}{\includegraphics{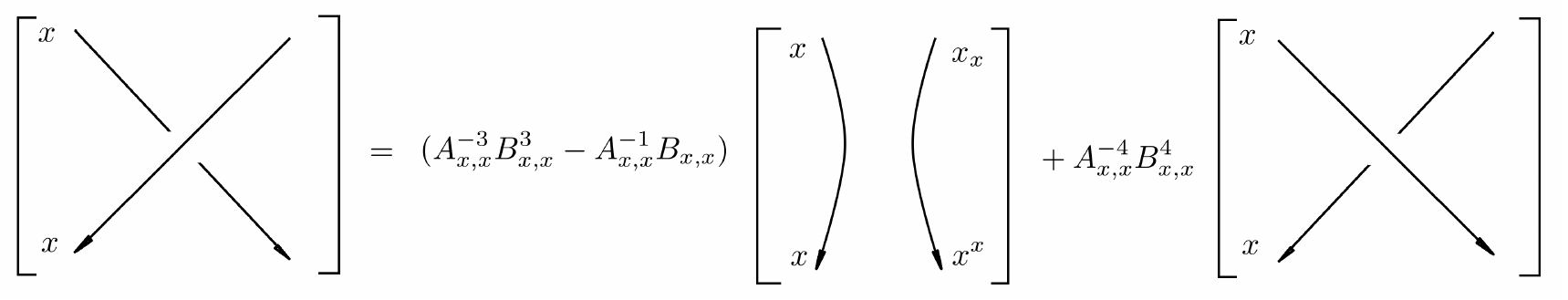}}\]
\end{proposition}

\begin{proof}
From the defining skein relations for $[\ ]$, the first identities of Lemma 
\ref{lem:1} and the fact that
$w=-A_{x,x}^2B_{x,x}^{-1}$ we have
\[\includegraphics{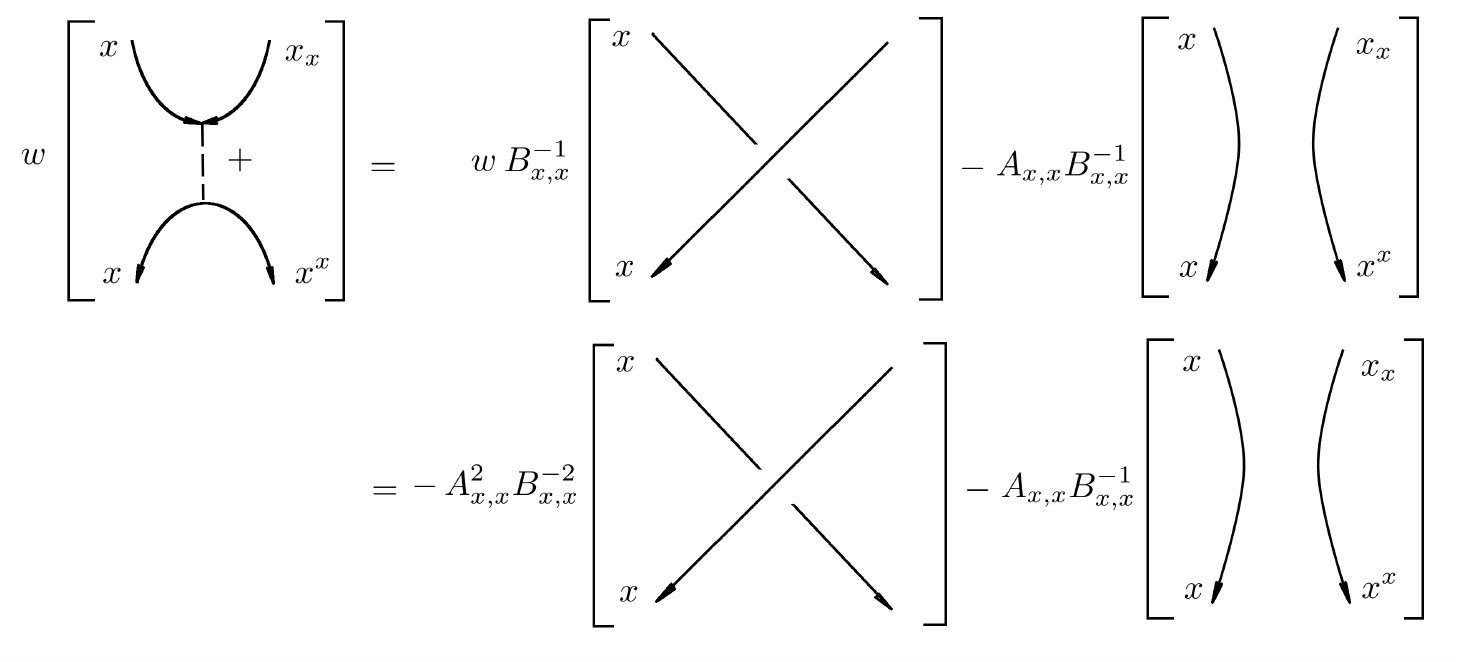}\]
and
\[\includegraphics{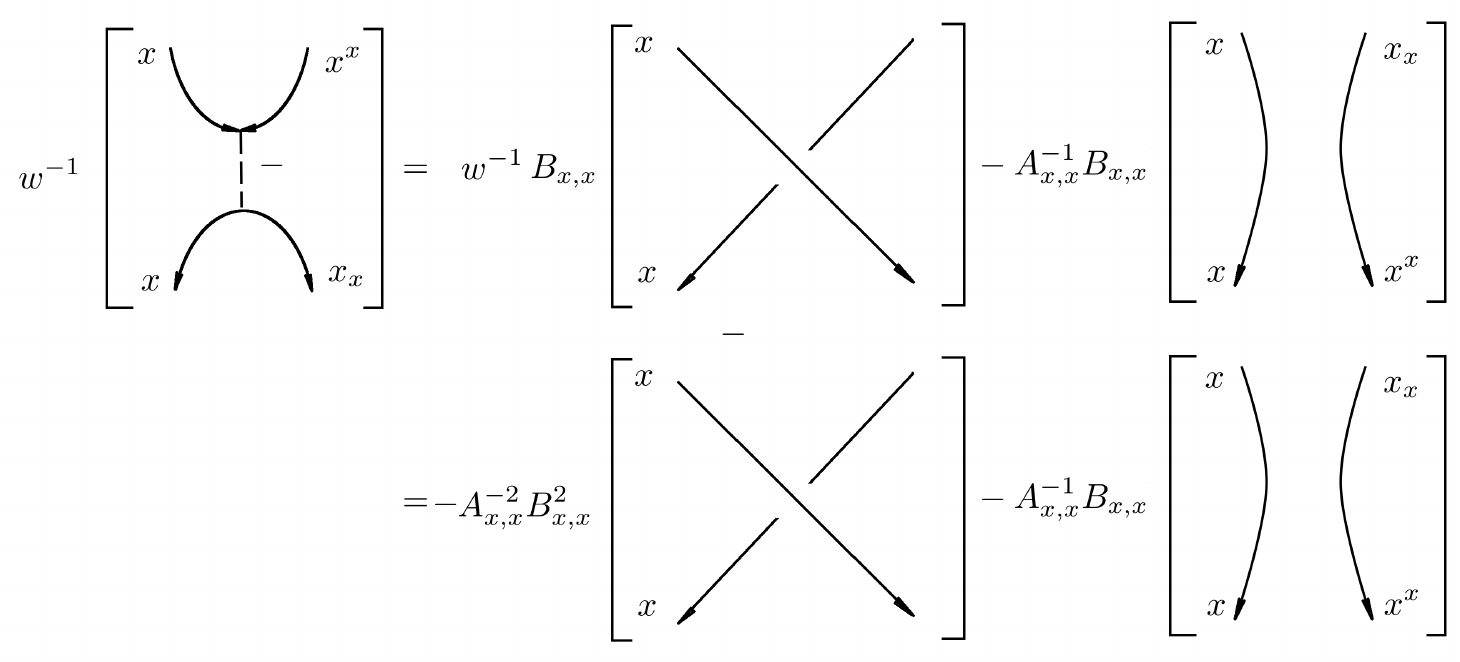}\]
Then substituting in the second equation of Lemma \ref{lem:1} above, we
obtain the result.
\end{proof}

We can sometimes use this skein relation to simplify the calculation of $[D]$ 
for diagrams with monochromatic crossings by reducing such diagrams to
unknots or unlinks.

\begin{example}
Let us compute $[D]$ for the biquandle coloring of the knot diagram
$\raisebox{-0.6in}{\scalebox{0.7}{\includegraphics{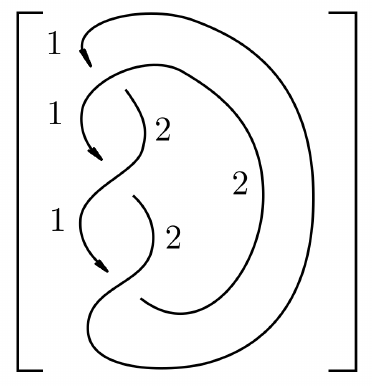}}}$
with a generic biquandle bracket 
in three ways: the state sum method, using trace diagrams with proposition
\ref{ns} and using the Homflypt-style skein relation from proposition
\ref{homfly}. 

With the state-sum method, there are eight states:
\[\scalebox{0.9}{\includegraphics{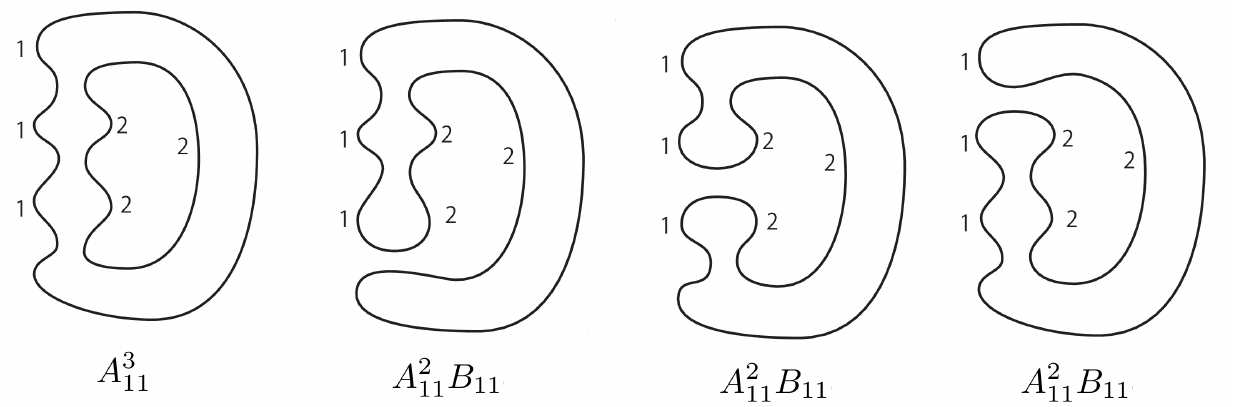}}\]
\[\scalebox{0.9}{\includegraphics{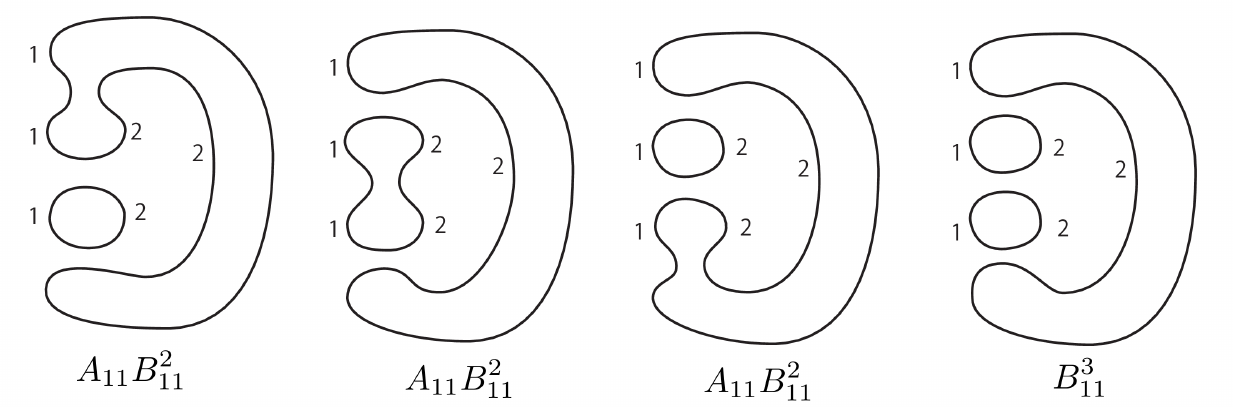}}\]
and we have
\begin{eqnarray*}
\raisebox{-0.6in}{\scalebox{0.8}{\includegraphics{sn-no1-60.pdf}}}
& = & A_{11}^{3}\delta^2 w^{-3} + A_{11}^{2}B_{11}\delta w^{-3}+A_{11}^{2}B_{11}\delta w^{-3}+A_{11}B_{11}^{2}\delta^{2}w^{-3}\\
& & +A_{11}^{2}B_{11} \delta w^{-3}+A_{11}B_{11}^{2}\delta^{2}w^{-3}+A_{11}B_{11}^{2}\delta^{2}w^{-3}+B_{11}^{3}\delta^{3}w^{-3}\\
& = &  3A_{11}^{2}B_{11}\delta w^{-3}+ (A_{11}^{3}+3A_{11}B_{11}^{2})\delta^{2} w^{-3}+B_{11}^{3}\delta^{3}w^{-3}\\
& = &  3A_{11}^{2}B_{11}(-A_{11}^{-1}B_{11}-A_{11}B_{11}^{-1})w^{-3}+ (A_{11}^{3}+3A_{11}B_{11}^{2})(-A_{11}^{-1}B_{11}-A_{11}B_{11}^{-1})^{2}w^{-3}\\
& &+B_{11}^{3}(-A_{11}^{-1}B_{11}-A_{11}B_{11}^{-1})^{3}w^{-3}\\
& = & -A_{11}^{-1}B_{11}^{1}-A_{11}^{-3}B_{11}^{3}-A_{11}^{-5}B_{11}^5+A_{11}^{-9}B_{11}^9.
\end{eqnarray*}

Using trace diagrams lets us reduce the number of diagrams we need:
\begin{eqnarray*}
\raisebox{-0.6in}{\scalebox{0.8}{\includegraphics{sn-no1-60.pdf}}}
& = & A_{11}\raisebox{-0.6in}{\scalebox{0.8}{\includegraphics{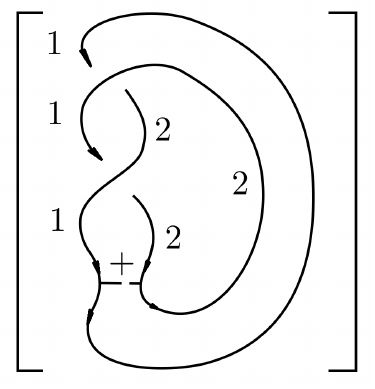}}}
+B_{11}\raisebox{-0.6in}{\scalebox{0.8}{\includegraphics{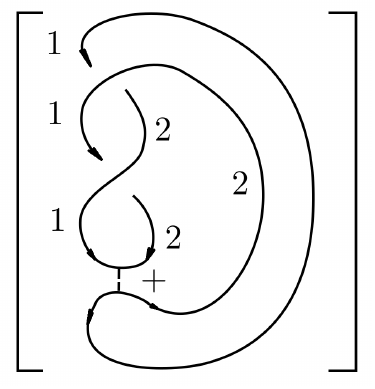}}} \\
& = & A_{11}^2\raisebox{-0.6in}{\scalebox{0.8}{\includegraphics{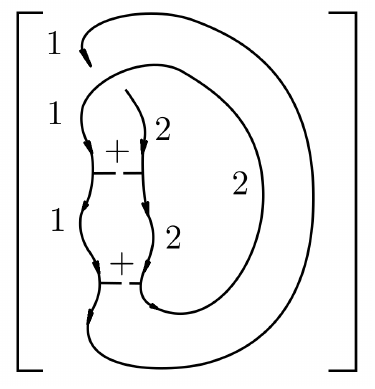}}}
+A_{11}B_{11}\raisebox{-0.6in}{\scalebox{0.8}{\includegraphics{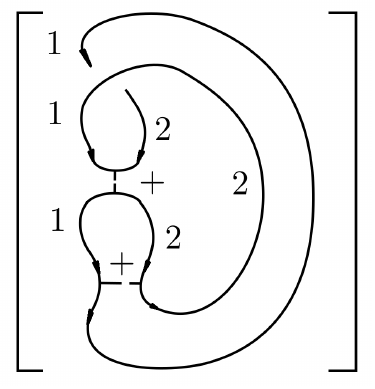}}}
+B_{11}(A_{11}+\delta B_{11})^2\delta w^{-3} \\
& = & A_{11}^2(\delta A_{11}+B_{11})\delta w^{-3}
+A_{11}B_{11}(A_{11}+\delta B_{11})\delta w^{-3}
+B_{11}(A_{11}+\delta B_{11})^2\delta w^{-3} \\
& = & (A_{11}^2(\delta A_{11}+B_{11})
+A_{11}B_{11}(A_{11}+\delta B_{11})
+B_{11}(A_{11}+\delta B_{11})^2)\delta w^{-3} \\
& = & (A_{11}^2(-A_{11}^2B_{11}^{-1})
+A_{11}B_{11}(-A_{11}^{-1}B_{11}^2)
+B_{11}(-A_{11}^{-1}B_{11}^2)^2)\delta w^{-3} \\
& = & (-A_{11}^4B_{11}^{-1}-B_{11}^3+A_{11}^{-2}B_{11}^5)\delta w^{-3} \\
& = & (-A_{11}^4B_{11}^{-1}-B_{11}^3+A_{11}^{-2}B_{11}^5)(-A_{11}^{-6}B_{11}^3)\delta \\
& = & (A_{11}^{-2}B_{11}^{2}+A_{11}^{-6}B_{11}^6-A_{11}^{-8}B_{11}^8)\delta \\
& = & -A_{11}^{-1}B_{11}^{1}-A_{11}^{-7}B_{11}^7+A_{11}^{-9}B_{11}^9-A_{11}^{-3}B_{11}^{3}-A_{11}^{-5}B_{11}^5+A_{11}^{-7}B_{11}^7 \\
& = & -A_{11}^{-1}B_{11}^{1}-A_{11}^{-3}B_{11}^{3}-A_{11}^{-5}B_{11}^5+A_{11}^{-9}B_{11}^9.
\end{eqnarray*}

Finally, using the Homflypt-style skein relation from proposition \ref{homfly},
we have
\begin{eqnarray*}
\raisebox{-0.6in}{\scalebox{0.8}{\includegraphics{sn-no1-60.pdf}}}
& = & (A_{11}^{-3}B_{11}^3-A_{11}^{-1}B_{11})
\raisebox{-0.6in}{\scalebox{0.8}{\includegraphics{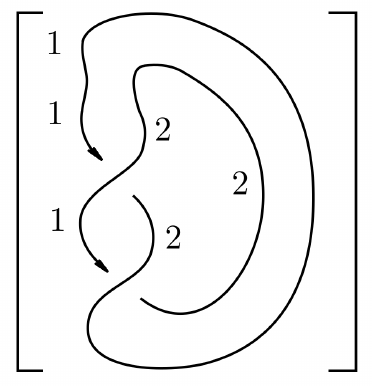}}}
+A_{11}^{-4}B_{11}^4
\raisebox{-0.6in}{\scalebox{0.8}{\includegraphics{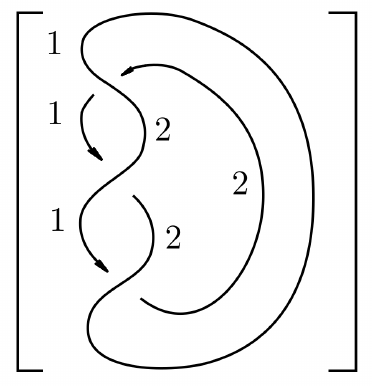}}}\\
& = & (A_{11}^{-3}B_{11}^3-A_{11}^{-1}B_{11})^2
\raisebox{-0.6in}{\scalebox{0.8}{\includegraphics{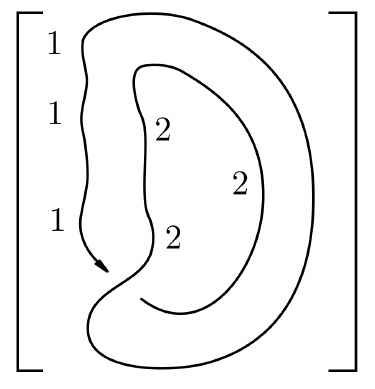}}}
+(A_{11}^{-7}B_{11}^7-A_{11}^{-5}B_{11}^4)
\raisebox{-0.6in}{\scalebox{0.8}{\includegraphics{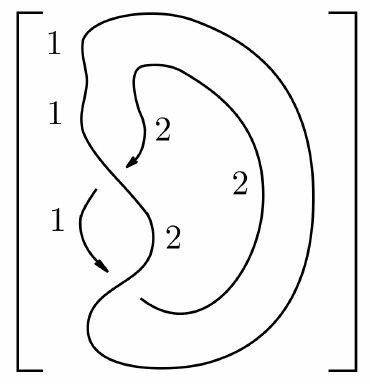}}}\\
& & +A_{11}^{-4}B_{11}^4\delta\\
&= & (A_{11}^{-3}B_{11}^3-A_{11}^{-1}B_{11})^2\delta
+(A_{11}^{-7}B_{11}^7-A_{11}^{-5}B_{11}^4)\delta^2
+A_{11}^{-4}B_{11}^4\delta\\
&= & (A_{11}^{-6}B_{11}^6-2A_{11}^{-4}B_{11}^4+A_{11}^{-2}B_{11}^2
+(A_{11}^{-7}B_{11}^7-A_{11}^{-5}B_{11}^4)\delta
+A_{11}^{-4}B_{11}^4)\delta\\
&= & (A_{11}^{-6}B_{11}^6-A_{11}^{-4}B_{11}^4+A_{11}^{-2}B_{11}^2
-A_{11}^{-8}B_{11}^8+A_{11}^{-6}B_{11}^6-A_{11}^{-6}B_{11}^6+A_{11}^{-4}B_{11}^4)\delta\\
&= & (A_{11}^{-6}B_{11}^6+A_{11}^{-2}B_{11}^2-A_{11}^{-8}B_{11}^8)
(-A_{11}^{-1}B_{11}-A_{11}B_{11}^{-1})\\
&= & -A_{11}^{-7}B_{11}^7-A_{11}^{-3}B_{11}^3+A_{11}^{-9}B_{11}^9
-A_{11}^{-5}B_{11}^5-A_{11}^{-1}B_{11}^1+A_{11}^{-7}B_{11}^7 \\
& = & -A_{11}^{-1}B_{11}^{1}-A_{11}^{-3}B_{11}^{3}-A_{11}^{-5}B_{11}^5+A_{11}^{-9}B_{11}^9.
\end{eqnarray*}

\end{example}

Sometimes, instead of moving a strand over or under a trace, we might want 
to move a strand \textit{through} a trace; after all, in uncolored skein 
expansions without traces we are free to move strands through the sites of
smoothings to enable faster unknotting. Unfortunately, for a general 
biquandle-colored diagram, such moves are generally obstructed by the 
biquandle coloring except in certain cases. We will consider the case of 
monochromatic colorings: 

\begin{lemma}
If the three colors on the left hand side of a Reidemeister III
move are the same then switching any of the three crossing
signs does not change the biquandle colors on the semiarcs.
\end{lemma}

\begin{proof}
We observe that if the biquandle colors are the same $x\in X$ down the left 
side of the move, then the colors on the semiarcs in the middle column are 
all equal to $y=x^x=x_x$ and the colors down the right column are all
equal to $z=y^y=y_y$. Switching a crossing sign in a crossing on the left
replaces $x_x$ with $x^x$ and vice-versa, but these are both $y$ on one
side of the move and replaces $y_y$ with $y^y$ on the other, but again these
are both $z$.
\[\includegraphics{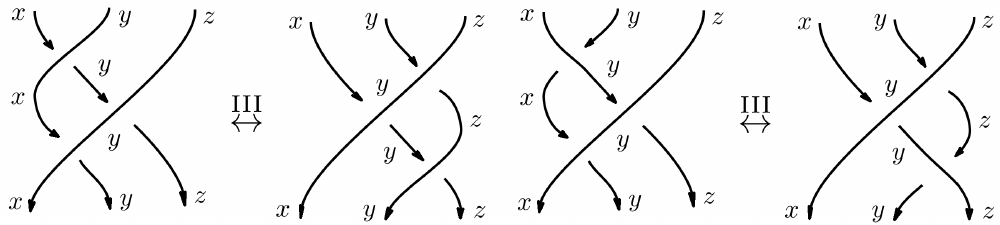}\]
\end{proof}

For the remainder of this section we will let $y=x^x=x_x$.
We identify four (eight if we count the moves with different trace signs 
separately) monochromatic trace pass-through moves:
\[\includegraphics{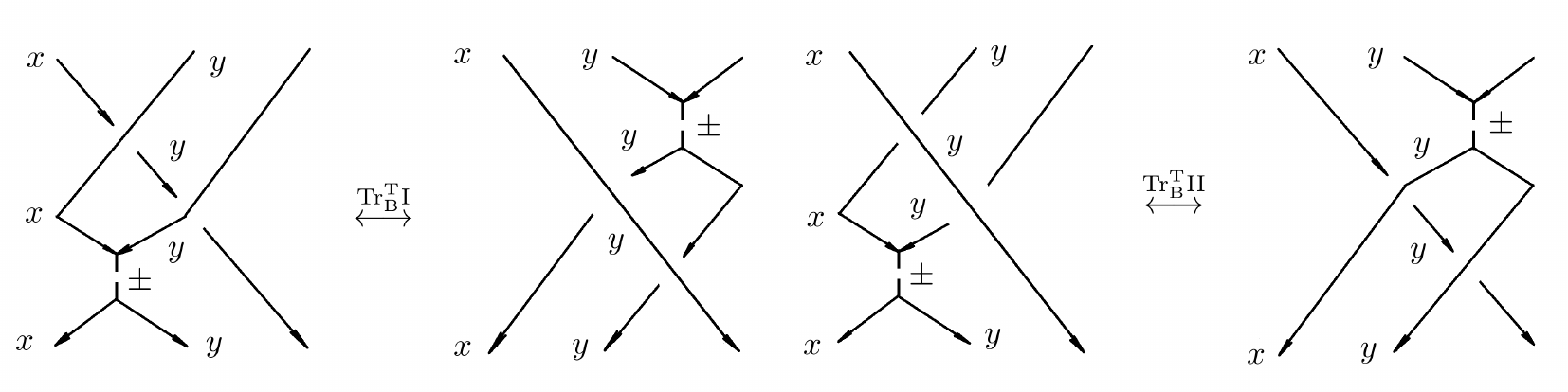}\]
\[\includegraphics{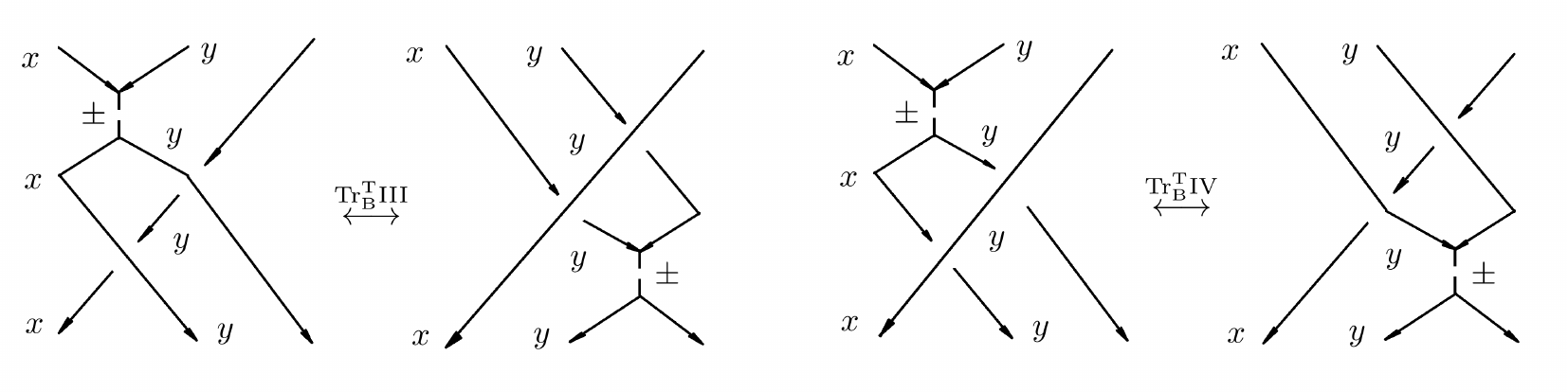}\]

Then we have
\begin{proposition}
A biquandle bracket $\beta$ satisfies the monochromatic trace pass-through 
moves if for all $x\in X$ and $y=x^x=x_x$ we have
\[A_{x,x}^2B_{y,y}^2=A_{y,y}^2B_{x,x}^2=1.\]
\end{proposition}

\begin{proof}
Comparing the coefficients after smoothing in move Tr$_B^T$I, we obtain the
requirements that 
\[\begin{array}{rcl}
A_{x,x}B_{y,y} & = & B_{x,x}^{-1}A_{y,y}^{-1}\\
A_{x,x}A_{y,y}+\delta B_{x,x}A_{y,y} + B_{x,x}B_{y,y} & = & 0\ \mathrm{and}\\
A_{x,x}^{-1}A_{y,y}^{-1}+\delta A_{x,x}^{-1}B_{y,y}^{-1} + B_{x,x}^{-1}B_{y,y}^{-1} & = & 0
\end{array}\]
Then the second equation reduces to 
\begin{eqnarray*}
\delta B_{x,x}A_{y,y} & = & -A_{x,x}A_{y,y}-B_{x,x}B_{y,y} \\
\delta & = & -A_{x,x}B_{x,x}^{-1}-A_{y,y}^{-1}B_{y,y} 
\end{eqnarray*}
and the third also reduces to
\begin{eqnarray*}
\delta A_{x,x}^{-1}B_{y,y}^{-1} & = & -A_{x,x}^{-1}A_{y,y}^{-1} - B_{x,x}^{-1}B_{y,y}^{-1} \\
\delta  & = & -A_{y,y}^{-1}B_{y,y} - A_{x,x}B_{x,x}^{-1}.
\end{eqnarray*}
Then since $\delta=-A_{x,x}B_{x,x}^{-1}-A_{x,x}^{-1}B_{x,x}$ this says
\begin{eqnarray*}
-A_{x,x}B_{x,x}^{-1}-A_{x,x}^{-1}B_{x,x} & = & -A_{x,x}B_{x,x}^{-1}-A_{y,y}^{-1}B_{y,y} \\
A_{x,x}^{-1}B_{x,x} & = & A_{y,y}^{-1}B_{y,y} \\
A_{y,y}B_{x,x} & = & A_{x,x}B_{y,y}
\end{eqnarray*}

Then combining this with the first condition 
$A_{x,x}B_{y,y} = B_{x,x}^{-1}A_{y,y}^{-1}$, we obtain
\[A_{y,y}^2B_{x,x}^2=1\]
and a similar computation writing $\delta=-A_{y,y}B_{y,y}^{-1}-A_{y,y}^{-1}B_{y,y}$ 
yields 
\[A_{x,x}^2B_{y,y}^2=1\]
as required. The other moves yield the same conditions.
\end{proof}

\section{Questions}\label{Q}

The conditions for over-adequacy and under-adequacy are suspiciously similar to
the quandle bracket conditions in \cite{NOR}. What algebraic properties of a
biquandle $X$ or of a ring $R$ are sufficient to guarantee over- and 
under-adequacy for all $X$-brackets over $R$? What is the topological meaning
of these conditions?

\bibliography{sn-no}{}
\bibliographystyle{abbrv}

\bigskip

\noindent
\textsc{Department of Mathematical Sciences \\
Claremont McKenna College \\
850 Columbia Ave. \\
Claremont, CA 91711} 

\

\noindent
\textsc{Department of Teacher Education \\
Shumei University \\
1-1 Daigaku-cho, Yachiyo \\
Chiba Prefecture 276-0003, Japan}

\end{document}